\documentclass[11pt, reqno]{amsart}
\usepackage{amssymb,mathrsfs,graphicx,extpfeil}
\usepackage{epsfig}
\usepackage{indentfirst, latexsym, amssymb, enumerate,amsmath,graphicx}
\usepackage{float}
\usepackage{colortbl}
\usepackage{epsfig,subfigure}
\usepackage{mathtools}
\newtheorem{theorem}{Theorem}[section]
\newtheorem{lemma}{Lemma}[section]

\newtheorem{remark}{Remark}[section]

\newtheorem{definition}{Definition}[section]

\def\charf {\mbox{{\text 1}\kern-.30em {\text l}}}
\title[Synchronization of the Kuramoto model with inertia and frustration]{Exponential synchronization of the Kuramoto model with inertia and frustration under locally coupled network}

\author[Zhu]{Tingting Zhu \textsuperscript{1}}
\author[Zhang]{Xiongtao Zhang \textsuperscript{2,}\textsuperscript{*}}
\thanks{$1$ School of Artificial Intelligence and Big Data, Hefei University, Hefei, China (ttzhud201880016@163.com)}
\thanks{$2$ School of Mathematics and Statistics, Wuhan University, Wuhan, China (zhangxt@whu.edu.cn)}
\thanks{$^*$ Corresponding author. }

\begin{document}
%%%%%%%%%%%%%%%%

\date{\today}


\begin{abstract}
We study the collective synchronized behavior of the Kuramoto model with inertia and frustration effects on a connected and symmetric network. We aim to establish sufficient frameworks for achieving complete frequency synchronization, taking into account initial configuration, small inertia and frustration, and large coupling strength. More precisely, we first demonstrate that the phase diameter will be uniformly bounded by a small value after a finite time. Then we prove that the frequency diameter exhibits exponential decay to zero. Our approach relies on a careful construction of energy functionals, which effectively control the dissipation of phase and frequency diameters.
\end{abstract}
\maketitle 
\centerline{\date}

%\tableofcontents
\section{Introduction}\label{sec:1}
\vspace{0.5cm}
The occurrence of collective synchronized behaviors in large oscillatory systems is ubiquitous in the real world. For instance, synchronous flashing of fireflies, heart beat generation, and the Josephson junction arrays \cite{A-B-P-R05,B-B66,P75,P-R-K01, S00, W-C-S96}. 
Due to their potential applications, synchronous phenomena have attracted significant interest in diverse scientific communities such as biology, physics, and engineering. The investigation of synchronization from a mathematical perspective was initially undertaken by Winfree and Kuramoto \cite{K75, W67}, in which the Kuramoto model was introduced. This model has became a representative model for describing synchronization and has been extensively studied \cite{C-H-J-K12, D-X13, H-K-P15, H-K-R}. To better simulate the realistic systems, researchers have introduced additional physical effects into the original Kuramoto model, including the frustration, inertia, time delay, and network topology \cite{D-H-K20, Do-B12, E91, H-K-L14, H-K-P18, H-K-Z18, H-C-Y-S99, S-K86,  Ta-L-O97, T-L-O97}.

 In this paper, we aim to investigate the Kuramoto model with inertia and frustration on a locally coupled network. To fix the idea, we consider a symmetric and connected network, which can be registered by a weighted graph $\mathcal{G} = (\mathcal{V}, E, A)$. Here $\mathcal{V} = \{1,2, \cdots, N\}$ is the vertex set, $A = (a_{ij})$ is an $N \times N$ matrix whose element $a_{ij}$ denotes the communication weight flowing from vertex $j$ to $i$, and $E = \{(i,j) : a_{ij} > 0\} \subset \mathcal{V} \times \mathcal{V}$ is the edge set. %We assume $a_{ij} = a_{ji}$. 
For a given graph $\mathcal{G}$, we assume that the oscillators in the Kuramoto model are located at the vertices of the graph $\mathcal{G}$, with their interactions determined by the edges in $E$ and the communication weights in $A$. Let $\theta_i = \theta_i(t) \in \mathbb{R}$ be the phase of the $i$-th oscillator with natural frequency $\Omega_i$. Then, the dynamics of Kuramoto oscillators with inertia $m>0$ and frustration $\alpha \in (0, \frac{\pi}{2})$ is governed by the following second-order phase system:
\begin{equation}\label{sp_dynamic}
m \ddot{\theta}_i(t) + \dot{\theta}_i(t) = \Omega_i + \frac{K}{N} \sum_{l=1}^N a_{il} \sin (\theta_l(t)- \theta_i(t) + \alpha), \quad t > 0, \ i =1,2, \ldots, N,
\end{equation}
subject to the initial data:
\begin{equation}\label{initial}
(\theta_i, \omega_i)(0) = (\theta_{i0},\omega_{i0}),
\end{equation}
where $K>0$ represents the coupling strength and $\omega_i(t) = \dot{\theta}_i(t)$ is the instantaneous frequency.

In the absence of frustration ($\alpha = 0$), the inertial Kuramoto model was initially proposed  by Ermentrout \cite{E91} as an explanation for the slow adaptation to frequency synchronization in certain species of fireflies. Significant progress has been made in understanding the dynamics of this model under different network structures and initial conditions. For the special case of an all-to-all network, i.e., $a_{il}=1$ in \eqref{sp_dynamic}, the emergence of exponential synchronization was proven when the oscillators are initially contained within a half circle \cite{C-H-N13, C-H-Y11}. Later on, the complete frequency synchronization was obtained under more general initial conditions \cite{C-H-M18, C-L-H-X-Y14, H-J-K19}. In the case of a connected and symmetric network, the second-order gradient-like flow structure was employed to establish the emergence of synchronization \cite{C-L19, C-L-H-X-Y14, L-X-Y14}. We refer to \cite{C-H-M18, W-Q17, W-C21} for synchronization with more general settings. 
 However, researches on the dynamics of the non-zero frustration case ($\alpha \ne 0$) in the Kuramoto model are relatively limited. Authors in \cite{Ha-K-L14} provided several sufficient frameworks leading to the exponential frequency synchronization for the model on an all-to-all network. They further studied the uniqueness and well-ordering property of emergent phase-locked states \cite{L-H16}. We refer to \cite{D-Z-P-L-J18, D-B12} for more information. To the best of our knowledge, the dynamics of the Kuramoto oscillators with inertia and frustration on a locally coupled network remain relatively unexplored from a theoretic standpoint. 
 %{\color{red}For the initial configuration located in a quarter circle, the authors in \cite{D-B12} derived the exponential synchronization for second-order Kuramoto oscillators by the first-order system under the over-damped condition.} 
 At least, there are two basic and natural questions:

 \begin{itemize}
 	\item Will the emergence of frequency synchronization occur in \eqref{sp_dynamic}?
 	\item What is the convergence rate of the emergence of synchronization? 
 \end{itemize}

The problems outlined above can indeed pose significant challenges. The first major difficulty lies in deriving the dissipation of the phase diameter in \eqref{sp_dynamic}. In the all-to-all case, one may show well ordering structure such as in \cite{L-H16}, so that the diameter $D(\theta(t))=\theta_i-\theta_j$ for fixed $i,j$. This allows the diameter to be second-order smooth and its dynamics can be effectively described by a second-order equation as shown in \cite{Ha-K-L14}. However, establishing such a well-ordering structure for the model \eqref{sp_dynamic} on a network still remains an open problem. Consequently, the diameter is only Lipschitz continuous and can not be governed by a second order equation. The second main difficulty arises from the lack of a gradient-flow structure in the system \eqref{sp_dynamic} due to the presence of the frustration term. As a result, even if one can derive the boundedness of the phase diameter, it is impossible to employ the criteria in \cite{C-L19,  C-L-H-X-Y14, L-X-Y14} to yield the convergence to a synchronous state, let alone exponential convergence.

Then, our contributions in this paper are two folds. Firstly, we establish a suitable energy functional that depends on the relative phases and frequencies of the system. This energy functional will play an important role in  controlling the phase diameter (see \eqref{p_functional}). Then, our analysis on the dissipation of this energy functional reveals that the phase diameter will decrease significantly within a finite time (see Lemma \ref{p_bound}). Secondly, we prove the frequency synchronization and yield the exponential convergence rate. To accomplish this, we combine together the dynamics of the frequencies and accelerations (see \eqref{f_functional}) to obtain a Cucker-Smale type estimate, which eventually shows the exponential emergence of the frequency synchronization (see Theorem \ref{main}).    

%In this paper, we consider the Kuramoto system \eqref{sp_dynamic} with inertia and frustration on a connected and symmetric graph.  
%Since the presence of frustration in the sine coupling function of system \eqref{sp_dynamic} results in the lack of second-order gradient-like formulation, this causes some mathematical difficulties in analyzing the collective synchronized behaviors. For instance, the approach based on the gradient-like flow and conservation law applied in \cite{C-L-H-X-Y14, C-L19} can not be directly applied to our problem. We next briefly discuss our main result in this paper. To begin with, we construct an energy functional $\mathcal{E}_1(t)$ (see \eqref{p_functional}) in terms of the phase and frequency differences which is different from the structures of energy functions in \cite{C-L-H-X-Y14, C-L19}. Under sufficient conditions in terms of initial data, small inertia, small frustration and large coupling strength, we derive the first-order Gronwall-type inequality of the dynamics of the energy function. Furthermore, the dissipation of energy functional $\mathcal{E}_1(t)$ yields that the Kuramoto oscillators will concentrate into a region with length less than $\frac{\pi}{2}$ after some finite time (see Lemma \ref{p_bound}). Next, we similarly construct another proper energy functional $\mathcal{E}_2(t)$ (see \eqref{f_functional}) and prove that the frequency diameter decays to zero exponentially fast through the first-order differential inequality of functional $\mathcal{E}_2(t)$ (see Theorem \ref{main}). 

The rest of paper is organized as follows. In Section \ref{sec:2}, we introduce several basic concepts and estimates, and provide our sufficient frameworks and main result. In Section \ref{sec:3}, we show that the phase diameter of the ensemble of Kuramoto oscillators will be confined in a quarter circle after some finite time. In Section \ref{sec:4}, we present the complete frequency synchronization emerges at the exponential convergence rate. Section \ref{sec:5} is devoted to a brief summary of our main result and future discussion.
\newline
%In this paper, we consider the Kuramoto model \eqref{sp_dynamic} with inertia and frustration on a connected and symmetric network. 
%Our main contribution in this paper is presented in Theorem \ref{main} in Section \ref{sec:2}.\newline

\noindent {\bf{Notation:}} For convenience, we introduce the notation related to phase and frequency diameters, lower and upper bound of the communication weight as below

\begin{align*}
&\theta(t) := (\theta_1(t), \ldots, \theta_N(t)), \quad \omega(t) := (\omega_1(t), \ldots, \omega_N(t)), \quad \Omega := (\Omega_1, \ldots, \Omega_N),\\
&D(\theta(t)):= \max_{1 \le i, j \le N}|\theta_i(t) - \theta_j(t)|, \quad D(\omega(t)) := \max_{1 \le i,j \le N} |\omega_i(t) - \omega_j(t)|, \\
&D_{\Omega} := \sum_{i=1}^N \sum_{j=1}^N |\Omega_i - \Omega_j|^2,\quad a_u := \max_{(i,j) \in E} a_{ij}, \quad a_l := \min_{(i,j) \in E} a_{ij}.\\
\end{align*}

\section{Preliminaries and main result}\label{sec:2}
\setcounter{equation}{0}

In this section, we first introduce some basic concepts and related estimates which will be used throughout the paper. Then, we give our assumptions and present our main result.

\subsection{Basic concepts} In this part, we provide some concepts such as connectivity and synchronization for the second-order Kuramoto system \eqref{sp_dynamic}.
\begin{definition}
Let $\mathcal{G} = (\mathcal{V}, E, A)$ be a graph. Then
\begin{enumerate}
\item A path from $i$ to $j$ is a sequence $i_0, i_1, \ldots,i_k$ with $i_0 = i$ and $i_k = j$ such that each successive pair of vertices is an edge of graph $\mathcal{G}$. 

\item If there exists a path from $i$ to $j$, then vertex $j$ is said to be reachable from vertex $i$.

\item A graph $\mathcal{G}$ is said to be connected if each vertex can be reachable from any other vertex of $\mathcal{G}$.
\end{enumerate}
\end{definition}

\begin{definition}
Let $\theta = \theta(t)$ be a solution to Kuramoto model \eqref{sp_dynamic}. The Kuramoto model exhibits asymptotic complete-frequency synchronization iff the relative frequency differences tend to zero asymptotically:
\begin{equation*}
\lim_{t \to \infty} |\dot{\theta}_i(t) - \dot{\theta}_j(t)| = 0, \quad \forall  \ i \ne j.
\end{equation*}
\end{definition}

\subsection{Preparatory estimates}
Let's consider a connected graph $\mathcal{G}$ and we see that for any $(i,j) \in \mathcal{V} \times \mathcal{V}$, there exists a shortest path from $i$ to $j$, say
\begin{equation*}
i = i_0 \ \to \ i_1 \ \to \ \cdots \ \to i_{r_{ij}} = j, \quad (i_{l+1},i_{l}) \in E, \ l=0,1,\cdots, r_{ij} -1,
\end{equation*}
 Here, $r_{ij}$ denotes the distance between vertices $i$ and $j$, i.e., the length of the shortest path from vertex $i$ to $j$. We set
\begin{equation}\label{para_graph}
r := \max_{1 \le i,j \le N} r_{ij}, \quad \Lambda_1 := \frac{1}{1+r |E^c|},
\end{equation}
where $E^c$ is the complement of edge set $E$ in $\mathcal{V} \times \mathcal{V}$ and $|E^c|$ denotes its cardinality.
As we know the graph $\mathcal{G}$ may not be complete, we first provide an equivalence relation in the following lemma.

\begin{lemma}\label{pd_equiv}
Suppose that the graph $\mathcal{G} = (\mathcal{V},E, A)$ is connected, and let $\theta_i(t)$ and $\omega_i(t)$ be the phase and frequency of the Kuramoto oscillator located at the vertex $i$ at time $t$, respectively. Then we have
\begin{equation*}
\begin{aligned}
&\Lambda_1 \sum_{i=1}^N \sum_{j=1}^N |\theta_i(t) - \theta_j(t)|^2 \le \sum_{(i,j) \in E} |\theta_i(t) - \theta_j(t)|^2 \le \sum_{i=1}^N \sum_{j=1}^N |\theta_i(t) -\theta_j(t)|^2,\\
&\Lambda_1 \sum_{i=1}^N \sum_{j=1}^N |\omega_i(t) - \omega_j(t)|^2 \le \sum_{(i,j) \in E} |\omega_i(t) - \omega_j(t)|^2 \le \sum_{i=1}^N \sum_{j=1}^N |\omega_i(t) -\omega_j(t)|^2,
\end{aligned}
\end{equation*}
where $\Lambda_1$ is a positive constant defined in \eqref{para_graph}.
\end{lemma}
\begin{proof}
We only prove the first inequality since the second one can be similarly dealt with. For the right part of the first inequality, it is obvious that
\begin{equation*}
\sum_{(i,j) \in E} |\theta_i(t) - \theta_j(t)|^2 \le \sum_{i=1}^N \sum_{j=1}^N |\theta_i(t) -\theta_j(t)|^2.
\end{equation*}
We next verify the left part of the first inequality. In fact, for $(k,l) \in E^c$, due to the connectivity of the graph, we can find a shortest path from vertex $k$ to $l$:
\begin{equation*}
k = i_0 \ \to \ i_1 \ \to \ \cdots \ \to \ i_p = l, \qquad (i_{s+1},i_s) \in E, \ s = 0, 1, \ldots, p-1,
\end{equation*}
where $p$ is the length of the shortest path from $k$ to $l$.
It is easy to see that
\begin{equation*}
|\theta_k - \theta_l| \le |\theta_{i_0} - \theta_{i_1}| + |\theta_{i_1} - \theta_{i_2}| + \cdots + |\theta_{i_{p-1}} - \theta_{i_p}|.
\end{equation*}
Then for $(k,l) \in E^c$, we have
\begin{equation*}
\begin{aligned}
|\theta_k(t) - \theta_l(t)|^2 &\le p[|\theta_{i_0}(t) - \theta_{i_1}(t)|^2 + |\theta_{i_1}(t) - \theta_{i_2}(t)|^2 + \cdots + |\theta_{i_{p-1}}(t) - \theta_{i_p}(t)|^2] \\
&\le r [|\theta_{i_0}(t) - \theta_{i_1}(t)|^2 + |\theta_{i_1}(t) - \theta_{i_2}(t)|^2 + \cdots + |\theta_{i_{p-1}}(t) - \theta_{i_p}(t)|^2] \\
&\le r \sum_{(i,j) \in E} |\theta_i(t) - \theta_j(t)|^2,
\end{aligned}
\end{equation*}
where the parameter $r$ is defined in \eqref{para_graph}.
This yields 
\begin{equation*}
\sum_{(k,l) \in E^c} |\theta_k(t) - \theta_l(t)|^2 \le |E^c| r \sum_{(i,j) \in E} |\theta_i(t) - \theta_j(t)|^2.
\end{equation*}
Thus, it follows that
\begin{equation*}
\sum_{(k,l) \in E^c} |\theta_k(t) - \theta_l(t)|^2 + \sum_{(k,l) \in E} |\theta_k(t) - \theta_l(t)|^2 \le (1 + |E^c| r) \sum_{(i,j) \in E} |\theta_i(t) - \theta_j(t)|^2.
\end{equation*}
Then, we derive that
\begin{equation*}
\frac{1}{1 + r|E^c| } \sum _{i=1}^N \sum_{j=1}^N |\theta_i(t) - \theta_j(t)|^2 \le \sum_{(i,j) \in E} |\theta_i(t) - \theta_j(t)|^2,
\end{equation*}
which completes the proof.
\end{proof}

\subsection{Main result} We now present our main result which states that the frequency synchronization occurs exponentially fast under sufficient conditions. For simplicity, we set $D_0$ and $D^\infty$ to be positive constants and $\mathcal{E}_1(t)$ to be an energy function below
\begin{align}\label{p_functional}
\mathcal{E}_1(t) = \ 2m \sum_{i,j=1}^N  (\theta_{i} - \theta_{j}) (\omega_{i} - \omega_{j}) + (1 - m\sqrt{K}) \sum_{i,j=1}^N(\theta_{i} - \theta_{j})^2 + 2m^2 \sum_{i,j=1}^N (\omega_i - \omega_j)^2.
\end{align}
Then we make following assumptions:

\begin{align}
&\textbf{Assumption $(\mathcal{A})$:}\notag\\
&\mathcal{E}_1(0) < \frac{1}{8} D_0^2, \quad D^\infty \in (0, \frac{\pi}{2}),\quad D_0 \in (0,\pi), \quad D^\infty < D_0, \label{Con_11}\\
&\sin \alpha < \frac{a_l \cos \alpha\cos D^\infty}{12a_u (1+ r|E^c|)\sin D^\infty}, \quad m< 1, \label{Con_12}\\
& \sqrt{K} >  \frac{(1+r|E^c|)}{a_l \cos \alpha}\max\left\{ \frac{D_0}{\sin D_0}, \frac{1}{\cos D^\infty}\right\}, \label{Con_13}\\
&m \sqrt{K} < \frac{1}{8}, \quad mK < \frac{a_l}{a_u^2 (1+r|E^c|) \cos \alpha} \min\left\{ \frac{\sin D_0}{36D_0},\frac{\cos D^\infty}{24}\right\}, \label{Con_14} \\
&K\sin \alpha < \frac{1}{8a_u\sin D^\infty},\label{Con_15}\\
&C_1 = \left[ \frac{3(1+|E^c|r)D_0}{K^{\frac{3}{2}}a_l \cos \alpha \sin D_0}+ 12\frac{m}{K^{\frac{1}{2}}} \right] D_\Omega + 12N^2K^{\frac{1}{2}} a_u^2 \sin^2 \alpha \left[\frac{(1+|E^c|r)D_0}{a_l \cos \alpha \sin D_0}  + 4mK\right] <\frac{1}{16} (D^\infty)^2. \label{Con_16} 
\end{align}

\begin{remark}
We briefly comment on Assumption $(\mathcal{A})$.
	\begin{enumerate}
		\item The set of parameters satisfying Assumption $(\mathcal{A})$ is nonempty. In fact, we may fix all the other parameters except for $m$, $K$ and $\alpha$. Then we set $m=\alpha=K^{-2}$ and let $K\gg 1$. Under this setting, Assumption $(\mathcal{A})$ is obviously fulfilled.
		\item If the oscillators are initially distributed in a half circle, i.e., $D(\theta(0))$ is close to $\pi$, our assumption $(\mathcal{A})$ may not work. This is the restriction of our method.     
	\end{enumerate}
\end{remark}
Under the Assumption $(\mathcal{A})$, our main result on exponential frequency synchronization is presented as follows.
\begin{theorem}\label{main}
Suppose the graph $\mathcal{G}$ is symmetric and connected, let $(\theta(t), \omega(t))$ be a solution to system \eqref{sp_dynamic} and \eqref{initial}, and suppose the Assumption $(\mathcal{A})$ is fulfilled. Then there exist a positive constant $C$
%and $D^\infty$, 
and a finite time $t_*$ such that 
\begin{equation*}
%D(\theta(t))\leq D^\infty<\frac{\pi}{2} \quad \text{and} \quad 
D(\omega(t)) \le C e^{- \frac{\sqrt{K}}{2}(t - t_*)}, \quad \text{for all} \ \   t \ge t_*.
\end{equation*}
\end{theorem}

\section{Uniformly small boundedness of phase diameter}\label{sec:3}
\setcounter{equation}{0}

We will show that the Kuramoto oscillators \eqref{sp_dynamic} are going to be trapped into a small region confined in a quarter circle under Assumption $(\mathcal{A})$. More precisely, we first introduce some basic estimates on the energy functional $\mathcal{E}_1$. Then, we provide the first-order Gronwall-type inequality to derive the desired results.

\subsection{Estimates on the energy functional $\mathcal{E}_1(t)$} Recall the energy functional defined in \eqref{p_functional}. The following lemma shows that the phase diameter can be bounded by $\mathcal{E}_1(t)$.
\begin{lemma}\label{p_positive}
Suppose the inertia and coupling strength satisfy
\begin{equation}\label{Con_1}
m\sqrt{K} < \frac{1}{8}.
\end{equation}
Then for any $t \ge 0$, we have
\begin{align*}
\mathcal{E}_1(t) \ge \frac{1}{8} \sum_{i=1}^N\sum_{j=1}^N(\theta_i(t) - \theta_j(t))^2 + \frac{2}{3} m^2 \sum_{i=1}^N \sum_{j=1}^N (\omega_i(t) - \omega_j(t))^2.
\end{align*}
Moreover, we have
\begin{align*}
D(\theta(t)) \le  \sqrt{8 \mathcal{E}_1(t)}, \quad t \ge 0.
\end{align*}
\end{lemma}
\begin{proof}
According to the formula
\begin{equation}\label{formula1}
2|x||y| \le x^2 + y^2, \quad x, y \in \mathbb{R},
\end{equation}
we have
\begin{equation*}
2m \sum_{i=1}^N \sum_{j=1}^N (\theta_i - \theta_j) (\omega_i - \omega_j)  \ge - \sum_{i=1}^N \sum_{j=1}^N \left[\frac{3}{4} (\theta_i - \theta_j)^2 + \frac{4}{3} m^2 (\omega_i - \omega_j)^2\right].
\end{equation*}
Therefore, we see from \eqref{p_functional} that
\begin{equation}\label{energy_pf}
\begin{aligned}
\mathcal{E}_1(t) &\ge \left(  \frac{1}{4} - m\sqrt{K} \right) \sum_{i=1}^N\sum_{j=1}^N(\theta_i - \theta_j)^2 + \frac{2}{3} m^2 \sum_{i=1}^N \sum_{j=1}^N (\omega_i - \omega_j)^2\\
&\ge \frac{1}{8} \sum_{i=1}^N\sum_{j=1}^N(\theta_i - \theta_j)^2 + \frac{2}{3} m^2 \sum_{i=1}^N \sum_{j=1}^N (\omega_i - \omega_j)^2,
\end{aligned}
\end{equation}
which complete the proof of the first part.

Moreover, it follows from \eqref{energy_pf} that
\begin{align*}
\mathcal{E}_1(t) \ge  \frac{1}{8} \sum_{i=1}^N\sum_{j=1}^N(\theta_i(t) - \theta_j(t))^2.
\end{align*}
This yields
\begin{align*}
(\theta_i(t) - \theta_j(t))^2 \le 8 \mathcal{E}_1(t).
\end{align*}
Thus, we obtain
\begin{align*}
D(\theta(t)) = \max_{1 \le i,j \le N} |\theta_i(t) - \theta_j(t)| \le \sqrt{8 \mathcal{E}_1(t)}.
\end{align*}
\end{proof}

\subsection{Differential inequality for $\mathcal{E}_1(t)$} In this part, we derive a Gronwall-type differential inequality for $\mathcal{E}_1$ along the flow \eqref{sp_dynamic}. Before this, we provide some a priori estimates.
\begin{lemma}\label{E1_p}
Suppose the graph $\mathcal{G}$ is symmetric and connected, and the phase configuration $\theta(t)=(\theta_1(t), \ldots,\theta_N(t))$ to system \eqref{sp_dynamic} at time $t$ satisfies
\begin{equation*}
D(\theta(t)) = \max_{1\le i,j \le N} |\theta_i(t) - \theta_j(t)| < D_0.
\end{equation*}
Then, we have the following estimates:

\begin{align}
&\sum_{i=1}^N\sum_{j=1}^N (\Omega_i - \Omega_j) \cdot 2(\theta_i - \theta_j) \label{CC-25}\\
&\le \frac{3(1+|E^c|r)D_0}{Ka_l \cos \alpha \sin D_0} \sum_{i=1}^N\sum_{j=1}^N (\Omega_i - \Omega_j)^2 + \frac{1}{3}Ka_l \cos \alpha  \frac{\sin D_0}{D_0}  \frac{1}{1 + |E^c| r} \sum_{i=1}^N\sum_{j=1}^N (\theta_i - \theta_j)^2,\notag\\
&\notag\\
&\frac{K}{N}\sin \alpha \sum_{i=1}^N\sum_{j=1}^N \sum_{l=1}^N [ a_{il} \cos (\theta_l - \theta_i) - a_{jl} \cos (\theta_l - \theta_j)] \cdot 2(\theta_i - \theta_j) \label{CC-26}\\
&\le  \frac{12N^2 K a_u^2 \sin^2 \alpha (1+|E^c|r)D_0}{a_l \cos \alpha \sin D_0} + \frac{1}{3}Ka_l \cos \alpha  \frac{\sin D_0}{D_0}  \frac{1}{1 + r|E^c| } \sum_{i=1}^N\sum_{j=1}^N (\theta_i - \theta_j)^2,\notag\\
&\notag\\
&\frac{K}{N}\cos \alpha \sum_{i=1}^N\sum_{j=1}^N \sum_{l=1}^N [ a_{il} \sin (\theta_l - \theta_i) - a_{jl} \sin (\theta_l - \theta_j)] \cdot 2(\theta_i - \theta_j) \label{CC-9}\\
&\le-2Ka_l \cos \alpha  \frac{\sin D_0}{D_0}  \frac{1}{1 + r|E^c| } \sum _{i=1}^N \sum_{j=1}^N |\theta_i - \theta_j|^2.\notag
\end{align}

\end{lemma}
\begin{proof} 
%\textbf{(1)} We apply the formula \eqref{formula1} to obtain
%\begin{align*}
%&\sum_{i=1}^N\sum_{j=1}^N (\Omega_i - \Omega_j) \cdot 2(\theta_i - \theta_j) \\
%&= \sum_{i=1}^N\sum_{j=1}^N 2 \sqrt{\frac{1}{\frac{1}{3}Ka_l \cos \alpha  \frac{\sin D_0}{D_0}  \frac{1}{1 + |E^c| r}}} (\Omega_i - \Omega_j) \cdot \sqrt{\frac{1}{3}K a_l\cos \alpha  \frac{\sin D_0}{D_0}  \frac{1}{1 + |E^c| r}} (\theta_i - \theta_j) \\
%&\le \frac{3(1+|E^c|r)D_0}{Ka_l \cos \alpha \sin D_0} \sum_{i=1}^N\sum_{j=1}^N (\Omega_i - \Omega_j)^2 + \frac{1}{3}Ka_l \cos \alpha  \frac{\sin D_0}{D_0}  \frac{1}{1 + |E^c| r} \sum_{i=1}^N\sum_{j=1}^N (\theta_i - \theta_j)^2. 
%\end{align*}
%
%\noindent \textbf{(2)} We use the fact $\cos x \le 1$ and formula \eqref{formula1} to get
%\begin{align*}
%&\frac{K}{N}\sin \alpha \sum_{i=1}^N\sum_{j=1}^N \sum_{l=1}^N [ a_{il} \cos (\theta_l - \theta_i) - a_{jl} \cos (\theta_l - \theta_j)] \cdot 2(\theta_i - \theta_j)\\
%&\le 4Ka_u \sin \alpha  \sum_{i=1}^N \sum_{j=1}^N |\theta_i - \theta_j| \\
%&= \sum_{i=1}^N \sum_{j=1}^N 2 \cdot 2Ka_u \sin \alpha \sqrt{\frac{1}{\frac{1}{3}Ka_l \cos \alpha  \frac{\sin D_0}{D_0}  \frac{1}{1 + |E^c| r}}} \cdot \sqrt{\frac{1}{3}K a_l\cos \alpha  \frac{\sin D_0}{D_0}  \frac{1}{1 + |E^c| r}} (\theta_i - \theta_j)\\
%&\le \frac{12N^2 K a_u^2 \sin^2 \alpha (1+|E^c|r)D_0}{a_l \cos \alpha \sin D_0} + \frac{1}{3}Ka_l \cos \alpha  \frac{\sin D_0}{D_0}  \frac{1}{1 + |E^c| r} \sum_{i=1}^N\sum_{j=1}^N (\theta_i - \theta_j)^2.
%\end{align*}
%\noindent \textbf{(3)} 
We will only show the proof of \eqref{CC-9}, since \eqref{CC-25} and \eqref{CC-26} are simply derived by applying \eqref{formula1}. For convenience, we set
\begin{equation*}
\mathcal{I}_2 = \frac{K}{N}\cos \alpha \sum_{i=1}^N\sum_{j=1}^N \sum_{l=1}^N [ a_{il} \sin (\theta_l - \theta_i) - a_{jl} \sin (\theta_l - \theta_j)] \cdot 2(\theta_i - \theta_j).
\end{equation*}
Now for $\mathcal{I}_2$, applying the symmetric property of the network, we have the following reduction:
\begin{equation}\label{D-1}
\begin{aligned} 
\mathcal{I}_2 
%&= \frac{K}{N}\cos \alpha \sum_{i=1}^N\sum_{j=1}^N \sum_{l=1}^N [ a_{il} \sin (\theta_l(t) - \theta_i(t)) - a_{jl} \sin (\theta_l(t) - \theta_j(t))] \cdot 2(\theta_i(t) - \theta_j(t)) \\
&= \frac{K}{N}\cos \alpha \sum^N_{\substack{i,j=1 \\ i \ne j}} \sum_{l=1}^N [ a_{il} \sin (\theta_l - \theta_i) - a_{jl} \sin (\theta_l - \theta_j)] \cdot 2(\theta_i - \theta_j) \\
&= \frac{K}{N}\cos \alpha \left\{ \sum^N_{\substack{i,j=1 \\ i \ne j}} \sum_{l = i,j} + \sum^N_{\substack{i,j=1 \\ i \ne j}} \sum^N_{\substack{l=1 \\ l \ne i,j}}\right\} [ a_{il} \sin (\theta_l - \theta_i) - a_{jl} \sin (\theta_l - \theta_j)] \cdot 2(\theta_i - \theta_j) \\
%&=  \frac{K}{N}\cos \alpha \sum^N_{\substack{i,j=1 \\ i \ne j}} \left[a_{ii} \sin (\theta_i(t) - \theta_i(t)) - a_{ji}\sin (\theta_i(t) - \theta_j(t)) \right.\\
%&\qquad \qquad \qquad \quad\left. + a_{ij} \sin (\theta_j(t) - \theta_i(t)) - a_{jj} \sin(\theta_j(t) - \theta_j(t))  \right] \cdot 2(\theta_i(t) - \theta_j(t))\\
%&+\frac{K}{N}\cos \alpha \sum^N_{\substack{i,j=1 \\ i \ne j}} \sum^N_{\substack{l=1 \\ l \ne i,j}} [ a_{il} \sin (\theta_l(t) - \theta_i(t)) - a_{jl} \sin (\theta_l(t) - \theta_j(t))] \cdot 2(\theta_i(t) - \theta_j(t)) \\
&= - \frac{2K}{N} \cos \alpha \sum^N_{\substack{i,j=1 \\ i \ne j}} a_{ji} \sin (\theta_i - \theta_j) \cdot 2(\theta_i - \theta_j) \\
&+\frac{K}{N}\cos \alpha  \sum^N_{\substack{i,j=1 \\ i \ne j}} \sum^N_{\substack{l=1 \\ l \ne i,j}} [a_{il} \sin (\theta_l - \theta_i)] \cdot 2(\theta_i - \theta_j) + \frac{K}{N}\cos \alpha  \sum^N_{\substack{i,j=1 \\ i \ne j}} \sum^N_{\substack{l=1 \\ l \ne i,j}} [- a_{jl} \sin (\theta_l - \theta_j)] \cdot 2(\theta_i - \theta_j).
\end{aligned}
\end{equation}
For the last two terms in above equality, we exchange the indices to see that
\begin{equation*}
\begin{aligned}
&\frac{K}{N}\cos \alpha  \sum^N_{\substack{i,j=1 \\ i \ne j}} \sum^N_{\substack{l=1 \\ l \ne i,j}} [a_{il} \sin (\theta_l - \theta_i)] \cdot 2(\theta_i - \theta_j) =\frac{K}{N}\cos \alpha \sum^N_{\substack{i,l=1 \\ i \ne l}} \sum^N_{\substack{j=1 \\  j \ne i,l}} [a_{ij} \sin (\theta_j - \theta_i)] \cdot 2(\theta_i - \theta_l),
\end{aligned}
\end{equation*}
and
\begin{equation*}
\begin{aligned}
&\frac{K}{N}\cos \alpha  \sum^N_{\substack{i,j=1 \\ i \ne j}} \sum^N_{\substack{l=1 \\ l \ne i,j}} [- a_{jl} \sin (\theta_l - \theta_j)] \cdot 2(\theta_i - \theta_j)= \frac{K}{N}\cos \alpha \sum^N_{\substack{l,j=1 \\ l \ne j}}   \sum^N_{\substack{i=1 \\ i \ne l,j}} [-a_{ji} \sin (\theta_i - \theta_j)] \cdot 2(\theta_l - \theta_j).
\end{aligned}
\end{equation*}
This implies that 
\begin{equation}\label{D-2}
\begin{aligned}
&\frac{K}{N}\cos \alpha  \sum^N_{\substack{i,j=1 \\ i \ne j}} \sum^N_{\substack{l=1 \\ l \ne i,j}} [a_{il} \sin (\theta_l - \theta_i)] \cdot 2(\theta_i - \theta_j) + \frac{K}{N}\cos \alpha  \sum^N_{\substack{i,j=1 \\ i \ne j}} \sum^N_{\substack{l=1 \\ l \ne i,j}} [- a_{jl} \sin (\theta_l - \theta_j)] \cdot 2(\theta_i - \theta_j)\\
&= - \frac{K}{N} \cos \alpha \sum^N_{\substack{i,l=1 \\ i \ne l}} \sum^N_{\substack{j=1 \\  j \ne i,l}} [a_{ji} \sin (\theta_i - \theta_j)] \cdot 2(\theta_i - \theta_l) -\frac{K}{N}\cos \alpha \sum^N_{\substack{l,j=1 \\ l \ne j}}   \sum^N_{\substack{i=1 \\ i \ne l,j}} [a_{ji} \sin (\theta_i - \theta_j)] \cdot 2(\theta_l - \theta_j) \\
&= - \frac{2K}{N}\cos \alpha  \sum^N_{\substack{i,j=1 \\ i\ne j}} \sum^N_{\substack{l=1 \\ l \ne i,j}} a_{ji} \sin (\theta_i - \theta_j)(\theta_i - \theta_j)=- \frac{2(N-2)K}{N}\cos \alpha \sum^N_{\substack{i,j=1\\ i \ne j}} a_{ji} \sin (\theta_i - \theta_j) (\theta_i - \theta_j), 
\end{aligned}
\end{equation}
where we used the equivalence relation between the following indices 
\[\sum^N_{\substack{i,l=1 \\ i \ne l}} \sum^N_{\substack{j=1 \\  j \ne i,l}}, \quad \sum^N_{\substack{l,j=1 \\ l \ne j}}   \sum^N_{\substack{i=1 \\ i \ne l,j}}, \quad \sum^N_{\substack{i,j=1 \\ i\ne j}} \sum^N_{\substack{l=1 \\ l \ne i,j}}.\]
Then, we combine \eqref{D-1}, \eqref{D-2} and Lemma \ref{pd_equiv} to obtain
\begin{equation*}
\begin{aligned}
\mathcal{I}_2 &= - \frac{2K}{N} \cos \alpha \sum^N_{\substack{i,j=1 \\ i \ne j}} a_{ji} \sin (\theta_i - \theta_j) \cdot 2(\theta_i - \theta_j) - \frac{2(N-2)K}{N}\cos \alpha \sum^N_{\substack{i,j=1\\ i \ne j}} a_{ji} \sin (\theta_i - \theta_j) (\theta_i - \theta_j) \\
&= -2K \cos \alpha \sum^N_{i=1} \sum_{j=1}^N a_{ji} \sin (\theta_i - \theta_j) (\theta_i - \theta_j) = -2K \cos \alpha \sum_{(j,i) \in E}  a_{ji} \sin (\theta_i - \theta_j) (\theta_i - \theta_j) \\
&\le -2K \cos \alpha \sum_{(j,i) \in E} a_l \frac{\sin D_0}{D_0} |\theta_i - \theta_j|^2 \le -2Ka_l \cos \alpha  \frac{\sin D_0}{D_0}  \frac{1}{1 + |E^c| r} \sum _{i=1}^N \sum_{j=1}^N |\theta_i - \theta_j|^2,
\end{aligned}
\end{equation*}
where we exploit the decreasing property of function $\frac{\sin x}{x}$ in $x \in (0,\pi]$.
% and the assumption $D(\theta(t)) < D_0 < \pi$.
\end{proof}
Lemma \ref{E1_p} focuses on the estimates of relative phases. As \eqref{sp_dynamic} is a second order system, we also need estimates of relative frequencies in the following lemma. 

\begin{lemma}\label{E1_f}
Suppose the graph $\mathcal{G}$ is symmetric and connected, and let $(\theta(t),\omega(t))$ be a solution to system \eqref{sp_dynamic}.
Then, we have the following estimates:

\begin{align}
&\sum_{i=1}^N \sum_{j=1}^N 4m (\Omega_i - \Omega_j) \cdot (\omega_i - \omega_j)
\le 12m \sum_{i=1}^N \sum_{j=1}^N (\Omega_i - \Omega_j)^2 + \frac{m}{3} \sum_{i=1}^N \sum_{j=1}^N (\omega_i - \omega_j)^2,\label{D-3}\\
&\notag\\
&\frac{2mK}{N}\sin \alpha \sum_{i=1}^N \sum_{j=1}^N\sum_{l=1}^N [ a_{il} \cos (\theta_l - \theta_i) - a_{jl} \cos (\theta_l - \theta_j)] \cdot  2(\omega_i - \omega_j)\label{D-4}\\
&\le 48N^2m K^2 a_u^2 \sin^2 \alpha + \frac{m}{3} \sum_{i=1}^N \sum_{j=1}^N (\omega_i - \omega_j)^2,\notag\\
&\notag\\
&\frac{2mK}{N}\cos \alpha \sum_{i=1}^N \sum_{j=1}^N\sum_{l=1}^N [ a_{il} \sin (\theta_l - \theta_i) - a_{jl} \sin (\theta_l - \theta_j)] \cdot 2(\omega_i - \omega_j)\label{D-5}\\
&\le 12mK^2a_u^2 \cos^2 \alpha \sum_{i=1}^N \sum_{j=1}^N (\theta_i - \theta_j)^2 + \frac{m}{3} \sum_{i=1}^N \sum_{j=1}^N (\omega_i - \omega_j)^2\notag.
\end{align}

\end{lemma} 

\begin{proof} Similar to Lemma \ref{E1_p}, we will only prove \eqref{D-5}.
% It can be seen that
%\begin{equation}\label{C-28}
%\begin{aligned}
%& \sum_{i=1}^N \sum_{j=1}^N 4m (\Omega_i - \Omega_j) \cdot (\omega_i - \omega_j)=\sum_{i=1}^N \sum_{j=1}^N 2 \cdot 2m \sqrt{\frac{3}{m}} (\Omega_i - \Omega_j) \cdot \sqrt{\frac{m}{3}} (\omega_i - \omega_j)\\
%%&\le \sum_{i=1}^N \sum_{j=1}^N \left[ 4m^2 \frac{3}{m}(\Omega_i - \Omega_j) ^2 + \frac{m}{3}(\dot{\theta}_i(t) - \dot{\theta}_j(t))^2 \right]\\
%&\le 12m \sum_{i=1}^N \sum_{j=1}^N (\Omega_i - \Omega_j)^2 + \frac{m}{3} \sum_{i=1}^N \sum_{j=1}^N (\omega_i - \omega_j)^2.
%\end{aligned}
%\end{equation}
%
%\noindent \textbf{(2)} We use the fact $\cos x \le 1$ to get
%\begin{equation*}
%\begin{aligned}
%&\frac{2mK}{N}\sin \alpha \sum_{i=1}^N \sum_{j=1}^N\sum_{l=1}^N [ a_{il} \cos (\theta_l - \theta_i) - a_{jl} \cos (\theta_l - \theta_j)] \cdot  2(\omega_i - \omega_j)\\
%&\le 8mK a_u \sin \alpha \sum_{i=1}^N \sum_{j=1}^N |\omega_i - \omega_j|= \sum_{i=1}^N \sum_{j=1}^N 2 \cdot 4mK a_u \sin \alpha \sqrt{\frac{3}{m}}  \cdot \sqrt{\frac{m}{3}} |\omega_i - \omega_j|\\
%%&\le \sum_{i=1}^N \sum_{j=1}^N \left[ 16m^2K^2a_u^2 \sin^2 \alpha \frac{3}{m} + \frac{m}{3} (\dot{\theta}_i(t) - \dot{\theta}_j(t))^2 \right]\\
%&\le 48N^2m K^2 a_u^2 \sin^2 \alpha + \frac{m}{3} \sum_{i=1}^N \sum_{j=1}^N (\omega_i - \omega_j)^2.
%\end{aligned}
%\end{equation*}
For notational simplicity, we set
\begin{equation*}
\mathcal{J}_2 = \frac{2mK}{N}\cos \alpha \sum_{i=1}^N \sum_{j=1}^N\sum_{l=1}^N [ a_{il} \sin (\theta_l - \theta_i) - a_{jl} \sin (\theta_l - \theta_j)] \cdot 2(\omega_i - \omega_j).
\end{equation*}
Next, we have the following reduction of $\mathcal{J}_2$:
\begin{equation}\label{D-6}
\begin{aligned}
\mathcal{J}_2 
%&= \frac{K}{N}\cos \alpha \sum_{i=1}^N \sum_{j=1}^N\sum_{l=1}^N [ a_{il} \sin (\theta_l(t) - \theta_i(t)) - a_{jl} \sin (\theta_l(t) - \theta_j(t))] \cdot 2(\dot{\theta}_i(t) - \dot{\theta}_j(t))\\
&= \frac{2mK}{N}\cos \alpha \sum^N_{\substack{i,j=1 \\ i \ne j}}\sum_{l=1}^N [ a_{il} \sin (\theta_l - \theta_i) - a_{jl} \sin (\theta_l - \theta_j)] \cdot 2(\omega_i - \omega_j)\\
&= \frac{2mK}{N}\cos \alpha  \left\{ \sum^N_{\substack{i,j=1 \\ i \ne j}} \sum_{l = i,j}+ \sum^N_{\substack{i,j=1 \\ i \ne j}} \sum^N_{\substack{l=1 \\ l \ne i,j}}\right\} [ a_{il} \sin (\theta_l - \theta_i) - a_{jl} \sin (\theta_l - \theta_j)] \cdot 2(\omega_i - \omega_j)\\
%&= \frac{K}{N}\cos \alpha \sum^N_{\substack{i,j=1 \\ i \ne j}}  \left[ a_{ii} \sin (\theta_i(t) - \theta_i(t)) - a_{ji} \sin (\theta_i(t) - \theta_j(t))\right.\\
%&\qquad \qquad \qquad \quad+\left. a_{ij} \sin (\theta_j(t) - \theta_i(t)) - a_{jj} \sin (\theta_j(t) - \theta_j(t)) \right] \cdot 2(\dot{\theta}_i(t) - \dot{\theta}_j(t))\\
%&+\frac{K}{N}\cos \alpha \sum^N_{\substack{i,j=1 \\ i \ne j}} \sum^N_{\substack{l=1 \\ l \ne i,j}} [ a_{il} \sin (\theta_l(t) - \theta_i(t)) - a_{jl} \sin (\theta_l(t) - \theta_j(t))] \cdot 2(\dot{\theta}_i(t) - \dot{\theta}_j(t))\\
&= - \frac{4mK}{N} \cos \alpha \sum^N_{\substack{i,j=1 \\ i \ne j}} a_{ji} \sin (\theta_i - \theta_j) \cdot 2(\omega_i - \omega_j)\\
& +\frac{2mK}{N}\cos \alpha \sum^N_{\substack{i,j=1 \\ i \ne j}} \sum^N_{\substack{l=1 \\ l \ne i,j}} [ a_{il} \sin (\theta_l - \theta_i)] \cdot 2(\omega_i - \omega_j)\\
&+\frac{2mK}{N}\cos \alpha \sum^N_{\substack{i,j=1 \\ i \ne j}} \sum^N_{\substack{l=1 \\ l \ne i,j}} [- a_{jl} \sin (\theta_l - \theta_j)] \cdot 2(\omega_i - \omega_j).
\end{aligned}
\end{equation}
Then, we may apply similar arguments in Lemma \ref{E1_p} to get 
\begin{equation*}
\begin{aligned}
\mathcal{J}_2 = -4mK \cos \alpha \sum_{i=1}^N \sum_{j=1}^N a_{ji} \sin (\theta_i - \theta_j) (\omega_i - \omega_j).
\end{aligned}
\end{equation*}
We further exploit the inequality $|\sin x| \le |x|$ to obtain
\begin{equation*}
\begin{aligned}
\mathcal{J}_2  
%&=-4mK \cos \alpha \sum_{i=1}^N \sum_{j=1}^N a_{ji} \sin (\theta_i - \theta_j) (\omega_i - \omega_j)\\
&\le 4mKa_u \cos \alpha \sum_{i=1}^N \sum_{j=1}^N |\theta_i(t) - \theta_j(t)| |\omega_i(t) - \omega_j(t)|\\
&= \sum_{i=1}^N \sum_{j=1}^N 2 \cdot 2mKa_u \cos \alpha \sqrt{\frac{3}{m}} |\theta_i - \theta_j| \cdot \sqrt{\frac{m}{3}} |\omega_i(t) - \omega_j(t)| \\
&= 12mK^2a_u^2 \cos^2 \alpha \sum_{i=1}^N \sum_{j=1}^N (\theta_i - \theta_j)^2 + \frac{m}{3} \sum_{i=1}^N \sum_{j=1}^N (\omega_i - \omega_j)^2,
\end{aligned}
\end{equation*}
which finishes the proof.
\end{proof}

Now we are ready to provide a differential inequality for energy functional $\mathcal{E}_1$ in the whole time interval.
\begin{lemma}\label{func1_dynamic}
Let $(\theta(t), \omega(t))$ be a solution to system \eqref{sp_dynamic} and \eqref{initial} and suppose the assumptions in \eqref{Con_11}, \eqref{Con_13}, \eqref{Con_14}, \eqref{Con_16} hold.
%the system parameters satisfy
%\begin{equation}\label{Con_2}
%\begin{aligned}
%&\mathcal{E}_1(0) < \frac{1}{8} D_0^2, \quad D^\infty \in (0, \frac{\pi}{2}),\quad D_0 \in (0,\pi), \quad D^\infty < D_0, \quad mK < \frac{a_l \sin D_0}{36a_u^2 \cos \alpha (1+|E^c|r)D_0}, \quad m \sqrt{K} < \frac{1}{8} \\
%&\sqrt{K}a_l \cos \alpha  \frac{\sin D_0}{D_0}  \frac{1}{1 + |E^c| r} > 1,  \\
%&C_1 := \left[ \frac{3(1+|E^c|r)D_0}{K^{\frac{3}{2}}a_l \cos \alpha \sin D_0}+ 12\frac{m}{K^{\frac{1}{2}}} \right] D_\Omega + 12N^2K^{\frac{1}{2}} a_u^2 \sin^2 \alpha \left[\frac{(1+|E^c|r)D_0}{a_l \cos \alpha \sin D_0}  + 4mK\right] <\frac{1}{16} (D^\infty)^2,\\
%\end{aligned}
%\end{equation}
Then, we have
\begin{equation*}
\dot{\mathcal{E}}_1(t) \le \sqrt{K}C_1 - \sqrt{K} \mathcal{E}_1(t), \quad t \in [0,+\infty),
\end{equation*}
where $C_1$ is a positive constant given in \eqref{Con_16}:
\begin{equation*}
C_1= \left[ \frac{3(1+|E^c|r)D_0}{K^{\frac{3}{2}}a_l \cos \alpha \sin D_0}+ 12\frac{m}{K^{\frac{1}{2}}} \right] D_\Omega + 12N^2K^{\frac{1}{2}} a_u^2 \sin^2 \alpha \left[\frac{(1+|E^c|r)D_0}{a_l \cos \alpha \sin D_0}  + 4mK\right].
\end{equation*}

\end{lemma}
%\begin{remark} In fact, based on the assumption in \eqref{Con_2}
%\begin{equation*}
% mK < \frac{a_l \sin D_0}{36a_u^2 \cos \alpha (1+|E^c|r)D_0}, \quad \frac{a_l \sin D_0}{36\sqrt{K}a_u^2 \cos \alpha (1+|E^c|r)D_0} < \frac{1}{8},
%\end{equation*}
%we derive that
%\begin{equation*}
%m\sqrt{K} < \frac{1}{8}, \quad \frac{2a_l \sin D_0}{9a_u^2 \cos \alpha (1+|E^c|r)D_0} < \sqrt{K}
%\end{equation*}
%\end{remark}

\begin{proof}
We prove by continuity argument and split the proof into four steps. 

\noindent $\bullet$ \textbf{Step 1:}
First, we define a set
\begin{equation*}
\mathcal{T} = \{T> 0 | \mathcal{E}_1(t) < \frac{1}{8}D_0^2 \ \text{for} \ 0 \le t < T\}.
\end{equation*}
Since $\mathcal{E}_1(0) < \frac{1}{8} D_0^2$ and $\mathcal{E}_1(t)$ is continuous, there exists a constant $\delta >0$ such that
\begin{equation*}
\mathcal{E}_1(t) < \frac{1}{8} D_0^2, \quad \text{for} \ 0 \le t < \delta.
\end{equation*}
Thus, the set $\mathcal{T}$ is nonempty and then we define $T^* = \sup \mathcal{T}$.
%according to Lemma \ref{p_positive}, we have
%\begin{equation*}
%D(\theta(0)) \le\sqrt{8\mathcal{E}_1(0)} < D_0 < \pi,
%\end{equation*}
Next, we will verify $T^* = +\infty$ by contrary. In fact, suppose  $T^* < +\infty$, then we have
\begin{equation}\label{continuity}
\mathcal{E}_1(t) < \frac{1}{8} D_0^2, \ \forall \ 0 \le t < T^* \quad \text{and} \quad \mathcal{E}_1(T^*) = \frac{1}{8} D_0^2.
\end{equation}
Moreover, according to Lemma \ref{p_positive}, we see that
\begin{equation}\label{C-1-1}
D(\theta(t)) \le \sqrt{8\mathcal{E}_1(t)} < D_0 < \pi, \quad \text{for} \ t \in [0,T^*).
\end{equation}
%Based on the property \eqref{C-1-1}, we will estimate the dynamics of functional $\mathcal{E}_1(t)$ in the following steps.

\noindent $\bullet$ \textbf{Step 2:} It follows from system \eqref{sp_dynamic} that for any $1 \le i, j \le N$,
%\begin{equation}\label{C-1}
%\begin{aligned}
%m\ddot{\theta}_i + \dot{\theta}_i &= \Omega_i + \frac{K}{N} \sum_{l=1}^N a_{il} \sin (\theta_l - \theta_i + \alpha) \\
%%&= \Omega_i + \frac{K}{N} \sum_{l=1}^N a_{il} \left[ \sin(\theta_l(t) - \theta_i(t)) \cos \alpha + \cos (\theta_l(t) - \theta_i(t)) \sin \alpha\right]\\
%&= \Omega_i + \frac{K}{N}\sin \alpha \sum_{l=1}^N a_{il} \cos (\theta_l - \theta_i) + \frac{K}{N}\cos \alpha \sum_{l=1}^N a_{il} \sin (\theta_l - \theta_i),
%\end{aligned}
%\end{equation}
%and 
%\begin{equation}\label{C-2}
%\begin{aligned}
%m\ddot{\theta}_j + \dot{\theta}_j &= \Omega_j + \frac{K}{N}\sin \alpha \sum_{l=1}^N a_{jl} \cos (\theta_l - \theta_j) + \frac{K}{N}\cos \alpha \sum_{l=1}^N a_{jl} \sin (\theta_l - \theta_j). 
%\end{aligned}
%\end{equation}
%We subtract \eqref{C-1} by \eqref{C-2} to obtain
\begin{equation}\label{C-3}
\begin{aligned}
m (\ddot{\theta}_i - \ddot{\theta}_j) + (\dot{\theta}_i - \dot{\theta}_j)&= \Omega_i - \Omega_j + \frac{K}{N}\sin \alpha \sum_{l=1}^N [ a_{il} \cos (\theta_l - \theta_i) - a_{jl} \cos (\theta_l - \theta_j)] \\
&+ \frac{K}{N}\cos \alpha \sum_{l=1}^N [ a_{il} \sin (\theta_l - \theta_i) - a_{jl} \sin (\theta_l - \theta_j)].
\end{aligned}
\end{equation}
We multiply \eqref{C-3} by $2(\theta_i - \theta_j)$ to get
%\begin{equation}\label{C-4}
%\begin{aligned}
%&m (\ddot{\theta}_i - \ddot{\theta}_j) \cdot 2(\theta_i - \theta_j) + (\dot{\theta}_i - \dot{\theta}_j) \cdot 2(\theta_i - \theta_j)\\
%&= (\Omega_i - \Omega_j) \cdot 2(\theta_i - \theta_j) + \frac{K}{N}\sin \alpha \sum_{l=1}^N [ a_{il} \cos (\theta_l - \theta_i) - a_{jl} \cos (\theta_l - \theta_j)] \cdot 2(\theta_i - \theta_j)\\
%&+ \frac{K}{N}\cos \alpha \sum_{l=1}^N [ a_{il} \sin (\theta_l - \theta_i) - a_{jl} \sin (\theta_l - \theta_j)] \cdot 2(\theta_i - \theta_j).
%\end{aligned}
%\end{equation}
%Since
%\begin{equation}\label{C-5}
%\begin{aligned}
%\frac{d}{dt} (\theta_i - \theta_j)^2 = 2(\theta_i - \theta_j) (\dot{\theta}_i - \dot{\theta}_j),\quad
%\frac{d^2}{dt^2} (\theta_i(t) - \theta_j(t))^2 
%%&= \frac{d}{dt} [2(\theta_i(t) - \theta_j(t)) (\dot{\theta}_i(t) - \dot{\theta}_j(t))] \\
%= 2(\dot{\theta}_i - \dot{\theta}_j)^2 + 2 (\theta_i - \theta_j) (\ddot{\theta}_i - \ddot{\theta}_j),
%\end{aligned}
%\end{equation}
%it can be obtained from \eqref{C-4} and \eqref{C-5} that
\begin{equation}\label{C-6}
\begin{aligned}
&m\left[ \frac{d^2}{dt^2} (\theta_i - \theta_j)^2 -  2(\dot{\theta}_i - \dot{\theta}_j)^2\right] + \frac{d}{dt} (\theta_i - \theta_j)^2 \\
&= (\Omega_i - \Omega_j) \cdot 2(\theta_i - \theta_j) + \frac{K}{N}\sin \alpha \sum_{l=1}^N [ a_{il} \cos (\theta_l - \theta_i) - a_{jl} \cos (\theta_l - \theta_j)] \cdot 2(\theta_i - \theta_j)\\
&+ \frac{K}{N}\cos \alpha \sum_{l=1}^N [ a_{il} \sin (\theta_l - \theta_i) - a_{jl} \sin (\theta_l - \theta_j)] \cdot 2(\theta_i - \theta_j).
\end{aligned}
\end{equation}
Moreover, we sum over $i, j$ on both sides of \eqref{C-6} to get
\begin{equation}\label{C-7}
\begin{aligned}
&m \frac{d^2}{dt^2} \sum_{i=1}^N\sum_{j=1}^N(\theta_i - \theta_j)^2 -  2m\sum_{i=1}^N\sum_{j=1}^N(\dot{\theta}_i - \dot{\theta}_j)^2 + \frac{d}{dt} \sum_{i=1}^N\sum_{j=1}^N(\theta_i - \theta_j)^2 \\
&=\sum_{i=1}^N\sum_{j=1}^N (\Omega_i - \Omega_j) \cdot 2(\theta_i - \theta_j) + \frac{K}{N}\sin \alpha \sum_{i=1}^N\sum_{j=1}^N \sum_{l=1}^N [ a_{il} \cos (\theta_l - \theta_i) - a_{jl} \cos (\theta_l - \theta_j)] \cdot 2(\theta_i - \theta_j)\\
&+ \frac{K}{N}\cos \alpha \sum_{i=1}^N\sum_{j=1}^N \sum_{l=1}^N [ a_{il} \sin (\theta_l - \theta_i) - a_{jl} \sin (\theta_l - \theta_j)] \cdot 2(\theta_i - \theta_j).\\
%&=\sum_{i=1}^N\sum_{j=1}^N (\Omega_i - \Omega_j) \cdot 2(\theta_i - \theta_j) + \mathcal{I}_1 + \mathcal{I}_2.
\end{aligned}
\end{equation}
Then combining \eqref{C-7} and the estimates \eqref{CC-25}, \eqref{CC-26}, \eqref{CC-9} in Lemma \ref{E1_p}, we have the following estimate for $1 \le t < T^*$,
\begin{equation}\label{D-7}
\begin{aligned}
&m \frac{d^2}{dt^2} \sum_{i=1}^N\sum_{j=1}^N(\theta_i - \theta_j)^2 -  2m\sum_{i=1}^N\sum_{j=1}^N(\omega_i - \omega_j)^2 + \frac{d}{dt} \sum_{i=1}^N\sum_{j=1}^N(\theta_i - \theta_j)^2 \\
&\le \frac{3(1+|E^c|r)D_0}{Ka_l \cos \alpha \sin D_0} \sum_{i=1}^N\sum_{j=1}^N (\Omega_i - \Omega_j)^2 + \frac{1}{3}Ka_l \cos \alpha  \frac{\sin D_0}{D_0}  \frac{1}{1 + r|E^c| } \sum_{i=1}^N\sum_{j=1}^N (\theta_i - \theta_j)^2\\
&+ \frac{12N^2 K a_u^2 \sin^2 \alpha (1+|E^c|r)D_0}{a_l \cos \alpha \sin D_0} + \frac{1}{3}Ka_l \cos \alpha  \frac{\sin D_0}{D_0}  \frac{1}{1 + r|E^c| } \sum_{i=1}^N\sum_{j=1}^N (\theta_i - \theta_j)^2\\
&-2Ka_l \cos \alpha  \frac{\sin D_0}{D_0}  \frac{1}{1 + r|E^c| } \sum _{i=1}^N \sum_{j=1}^N (\theta_i - \theta_j)^2\\
&= \frac{3(1+|E^c|r)D_0}{Ka_l \cos \alpha \sin D_0} D_\Omega + \frac{12N^2 K a_u^2 \sin^2 \alpha (1+|E^c|r)D_0}{a_l \cos \alpha \sin D_0} \\
&- \frac{4}{3}Ka_l \cos \alpha  \frac{\sin D_0}{D_0}  \frac{1}{1 + r|E^c| } \sum _{i=1}^N \sum_{j=1}^N (\theta_i - \theta_j)^2.
\end{aligned} 
\end{equation}

\noindent $\bullet$ \textbf{Step 3:}
Next, 
%\begin{equation*}
%\begin{aligned}
%&m (\ddot{\theta}_i(t) - \ddot{\theta}_j(t)) + (\dot{\theta}_i(t) - \dot{\theta}_j(t))\\ &= \Omega_i - \Omega_j + \frac{K}{N}\sin \alpha \sum_{l=1}^N [ a_{il} \cos (\theta_l(t) - \theta_i(t)) - a_{jl} \cos (\theta_l(t) - \theta_j(t))] \\
%&+ \frac{K}{N}\cos \alpha \sum_{l=1}^N [ a_{il} \sin (\theta_l(t) - \theta_i(t)) - a_{jl} \sin (\theta_l(t) - \theta_j(t))],
%\end{aligned}
%\end{equation*}
we multiply \eqref{C-3} by $2(\dot{\theta}_i - \dot{\theta}_j)$ to gain the dissipation of relative frequencies. Actually, we have
%\begin{equation}\label{C-13}
%\begin{aligned}
%&m (\ddot{\theta}_i - \ddot{\theta}_j) \cdot 2(\dot{\theta}_i - \dot{\theta}_j) + (\dot{\theta}_i - \dot{\theta}_j) \cdot 2(\dot{\theta}_i - \dot{\theta}_j) \\
%&= (\Omega_i - \Omega_j) \cdot 2(\dot{\theta}_i - \dot{\theta}_j) +  \frac{K}{N}\sin \alpha \sum_{l=1}^N [ a_{il} \cos (\theta_l - \theta_i) - a_{jl} \cos (\theta_l - \theta_j)] \cdot  2(\dot{\theta}_i - \dot{\theta}_j)\\
%&+ \frac{K}{N}\cos \alpha \sum_{l=1}^N [ a_{il} \sin (\theta_l - \theta_i) - a_{jl} \sin (\theta_l - \theta_j)] \cdot 2(\dot{\theta}_i - \dot{\theta}_j).
%\end{aligned}
%\end{equation}
%According to 
%\begin{equation*}
%\frac{d}{dt} (\dot{\theta}_i - \dot{\theta}_j)^2 = 2 (\dot{\theta}_i - \dot{\theta}_j) (\ddot{\theta}_i - \ddot{\theta}_j),
%\end{equation*}
%it follows from \eqref{C-13} that
\begin{equation}\label{C-14}
\begin{aligned}
&m \frac{d}{dt} (\dot{\theta}_i - \dot{\theta}_j)^2 + 2 (\dot{\theta}_i - \dot{\theta}_j)^2\\
&= (\Omega_i - \Omega_j) \cdot 2(\dot{\theta}_i - \dot{\theta}_j) +  \frac{K}{N}\sin \alpha \sum_{l=1}^N [ a_{il} \cos (\theta_l - \theta_i) - a_{jl} \cos (\theta_l - \theta_j)] \cdot  2(\dot{\theta}_i - \dot{\theta}_j)\\
&+ \frac{K}{N}\cos \alpha \sum_{l=1}^N [ a_{il} \sin (\theta_l - \theta_i) - a_{jl} \sin (\theta_l - \theta_j)] \cdot 2(\dot{\theta}_i - \dot{\theta}_j).
\end{aligned}
\end{equation}
%Moreover, we sum over $i,j$ and further multiply by $2m$ on both sides of \eqref{C-14} to have
%\begin{equation}\label{C-15}
%\begin{aligned}
%&2m^2 \frac{d}{dt} \sum_{i=1}^N \sum_{j=1}^N (\omega_i - \omega_j)^2 + 4m \sum_{i=1}^N \sum_{j=1}^N (\omega_i - \omega_j)^2\\
%&=4m \sum_{i=1}^N \sum_{j=1}^N (\Omega_i - \Omega_j) (\omega_i - \omega_j) +  \frac{2mK}{N}\sin \alpha \sum_{i=1}^N \sum_{j=1}^N\sum_{l=1}^N [ a_{il} \cos (\theta_l - \theta_i) - a_{jl} \cos (\theta_l - \theta_j)] \cdot  2(\omega_i - \omega_j)\\
%&+ \frac{2mK}{N}\cos \alpha \sum_{i=1}^N \sum_{j=1}^N\sum_{l=1}^N [ a_{il} \sin (\theta_l - \theta_i) - a_{jl} \sin (\theta_l - \theta_j)] \cdot 2(\omega_i - \omega_j)\\
%%&= \sum_{i=1}^N \sum_{j=1}^N (\Omega_i - \Omega_j) \cdot 2(\dot{\theta}_i - \dot{\theta}_j) + \mathcal{J}_1 + \mathcal{J}_2.
%\end{aligned}
%\end{equation}
Recall $\dot{\theta}_i=\omega_i$. Then, we sum over $i,j$ in \eqref{C-14} and apply  the estimates \eqref{D-3}, \eqref{D-4}, \eqref{D-5} in Lemma \ref{E1_f} to obtain  
\begin{equation}\label{D-8}
\begin{aligned}
&2m^2 \frac{d}{dt} \sum_{i=1}^N \sum_{j=1}^N (\omega_i - \omega_j)^2 + 4m \sum_{i=1}^N \sum_{j=1}^N (\omega_i - \omega_j)^2\\
&\le 12m \sum_{i=1}^N \sum_{j=1}^N (\Omega_i - \Omega_j)^2 + \frac{m}{3} \sum_{i=1}^N \sum_{j=1}^N (\omega_i - \omega_j)^2+48N^2m K^2 a_u^2 \sin^2 \alpha + \frac{m}{3} \sum_{i=1}^N \sum_{j=1}^N (\omega_i - \omega_j)^2\\
&+12mK^2a_u^2 \cos^2 \alpha \sum_{i=1}^N \sum_{j=1}^N (\theta_i - \theta_j)^2 + \frac{m}{3} \sum_{i=1}^N \sum_{j=1}^N (\omega_i - \omega_j)^2\\
&= 12m D_\Omega + 48N^2m K^2 a_u^2 \sin^2 \alpha + 12mK^2a_u^2 \cos^2 \alpha \sum_{i=1}^N \sum_{j=1}^N (\theta_i - \theta_j)^2 + m\sum_{i=1}^N \sum_{j=1}^N (\omega_i - \omega_j)^2.
\end{aligned}
\end{equation}

\noindent \textbf{Step 4:} Now, we add \eqref{D-7} and \eqref{D-8} together to derive that for $0 \le t < T^*$,
\begin{equation}\label{D-9}
\begin{aligned}
&m \frac{d^2}{dt^2} \sum_{i=1}^N\sum_{j=1}^N(\theta_i - \theta_j)^2 -  2m\sum_{i=1}^N\sum_{j=1}^N(\omega_i - \omega_j)^2 + \frac{d}{dt} \sum_{i=1}^N\sum_{j=1}^N(\theta_i - \theta_j)^2\\
&+2m^2 \frac{d}{dt} \sum_{i=1}^N \sum_{j=1}^N (\omega_i - \omega_j)^2 + 4m \sum_{i=1}^N \sum_{j=1}^N (\omega_i - \omega_j)^2\\
&\le \frac{3(1+|E^c|r)D_0}{Ka_l \cos \alpha \sin D_0} D_\Omega + \frac{12N^2 K a_u^2 \sin^2 \alpha (1+|E^c|r)D_0}{a_l \cos \alpha \sin D_0} \\
&- \frac{4}{3}Ka_l \cos \alpha  \frac{\sin D_0}{D_0}  \frac{1}{1 + r|E^c| } \sum _{i=1}^N \sum_{j=1}^N (\theta_i - \theta_j)^2 +12m D_\Omega + 48N^2m K^2 a_u^2 \sin^2 \alpha \\
&+ 12mK^2a_u^2 \cos^2 \alpha \sum_{i=1}^N \sum_{j=1}^N (\theta_i - \theta_j)^2 + m\sum_{i=1}^N \sum_{j=1}^N (\omega_i - \omega_j)^2.\\
\end{aligned}
\end{equation}
Due to the assumption in $\eqref{Con_14}_2$, we have 
\begin{equation*}
\begin{aligned}
&\frac{1}{3}Ka_l \cos \alpha  \frac{\sin D_0}{D_0}  \frac{1}{1 + |E^c| r} - 12mK^2a_u^2 \cos^2 \alpha > 0.
%&mK < \frac{a_l}{a_u^2 \cos \alpha(1+r|E^c|)} \min\left\{ \frac{\sin D_0}{36D_0},\frac{\cos D^\infty}{24}\right\}, 
\end{aligned}
\end{equation*}
Thus, \eqref{D-9} can be further estimated and rewritten as below,
%\begin{equation}\label{D-10}
%\begin{aligned}
%&m \frac{d^2}{dt^2} \sum_{i=1}^N\sum_{j=1}^N(\theta_i - \theta_j)^2  + \frac{d}{dt} \sum_{i=1}^N\sum_{j=1}^N(\theta_i - \theta_j)^2 +2m^2 \frac{d}{dt} \sum_{i=1}^N \sum_{j=1}^N (\omega_i - \omega_j)^2 \\
%&\le \left[ \frac{3(1+|E^c|r)D_0}{Ka_l \cos \alpha \sin D_0}+ 12m \right] D_\Omega + 12N^2K a_u^2 \sin^2 \alpha \left[\frac{(1+|E^c|r)D_0}{a_l \cos \alpha \sin D_0}  + 4mK\right]\\
%&-Ka_l \cos \alpha  \frac{\sin D_0}{D_0}  \frac{1}{1 + r|E^c| } \sum _{i=1}^N \sum_{j=1}^N (\theta_i - \theta_j)^2-  m\sum_{i=1}^N\sum_{j=1}^N(\omega_i - \omega_j)^2, \quad 0 \le t < T^*.
%\end{aligned}
%\end{equation}
%This can be further transformed into the following form
\begin{equation}\label{D-10}
\begin{aligned}
&\frac{d}{dt} \left[ m \frac{d}{dt}\sum_{i=1}^N\sum_{j=1}^N(\theta_i - \theta_j)^2 + (1 - m\sqrt{K})\sum_{i=1}^N\sum_{j=1}^N(\theta_i - \theta_j)^2 + 2m^2 \sum_{i=1}^N \sum_{j=1}^N (\omega_i - \omega_j)^2  \right]\\
%&\le \left[ \frac{3(1+|E^c|r)D_0}{Ka_l \cos \alpha \sin D_0}+ 12m \right] D_\Omega + 12N^2K a_u^2 \sin^2 \alpha \left[\frac{(1+|E^c|r)D_0}{a_l \cos \alpha \sin D_0}  + 4mK\right]\\
%&-m\sqrt{K}  \frac{d}{dt} \sum_{i=1}^N\sum_{j=1}^N(\theta_i - \theta_j)^2 -Ka_l \cos \alpha  \frac{\sin D_0}{D_0}  \frac{1}{1 + r|E^c| } \sum _{i=1}^N \sum_{j=1}^N (\theta_i - \theta_j)^2-  m\sum_{i=1}^N\sum_{j=1}^N(\omega_i - \omega_j)^2\\
& \leq \left[ \frac{3(1+|E^c|r)D_0}{Ka_l \cos \alpha \sin D_0}+ 12m \right] D_\Omega + 12N^2K a_u^2 \sin^2 \alpha \left[\frac{(1+|E^c|r)D_0}{a_l \cos \alpha \sin D_0}  + 4mK\right] \\
&- \sqrt{K} \left[ m  \frac{d}{dt} \sum_{i=1}^N\sum_{j=1}^N(\theta_i - \theta_j)^2 + \sqrt{K}a_l \cos \alpha  \frac{\sin D_0}{D_0}  \frac{1}{1 + |E^c| r} \sum _{i=1}^N \sum_{j=1}^N (\theta_i - \theta_j)^2 \right.\\
&\left.\qquad\qquad+ \frac{m}{\sqrt{K}} \sum_{i=1}^N \sum_{j=1}^N (\omega_i - \omega_j)^2\right].\\
\end{aligned}
\end{equation}
Then owing to the assumptions in \eqref{Con_13} and $\eqref{Con_14}_2$, we know that
\begin{equation*}
\sqrt{K}a_l \cos \alpha  \frac{\sin D_0}{D_0}  \frac{1}{1 + |E^c| r} > 1> 1- m \sqrt{K}, \quad \frac{m}{\sqrt{K}} > 2m^2,
\end{equation*} 
which together with \eqref{D-10} immediately implies that for $0 \le t < T^*$,
\begin{equation*}
\begin{aligned}
&\frac{d}{dt} \left[ m \frac{d}{dt}\sum_{i=1}^N\sum_{j=1}^N(\theta_i - \theta_j)^2 + (1 - m\sqrt{K})\sum_{i=1}^N\sum_{j=1}^N(\theta_i - \theta_j)^2 + 2m^2 \sum_{i=1}^N \sum_{j=1}^N (\omega_i - \omega_j)^2  \right]\\
&\le \left[ \frac{3(1+|E^c|r)D_0}{Ka_l \cos \alpha \sin D_0}+ 12m \right] D_\Omega + 12N^2K a_u^2 \sin^2 \alpha \left[\frac{(1+|E^c|r)D_0}{a_l \cos \alpha \sin D_0}  + 4mK\right] \\
&- \sqrt{K} \left[m  \frac{d}{dt} \sum_{i=1}^N\sum_{j=1}^N(\theta_i - \theta_j)^2 + (1 - m\sqrt{K}) \sum_{i=1}^N\sum_{j=1}^N(\theta_i - \theta_j)^2 + 2m^2 \sum_{i=1}^N \sum_{j=1}^N (\omega_i - \omega_j)^2 \right]. 
\end{aligned}
\end{equation*}
According to \eqref{p_functional}, this means that for $t \in [0, T^*)$,
%\begin{equation}\label{C-30}
%\begin{aligned}
%\dot{\mathcal{E}}_1(t) \le \left[ \frac{3(1+|E^c|r)D_0}{Ka_l \cos \alpha \sin D_0}+ 12m \right] D_\Omega + 12N^2K a_u^2 \sin^2 \alpha \left[\frac{(1+|E^c|r)D_0}{a_l \cos \alpha \sin D_0}  + 4mK\right] - \sqrt{K} \mathcal{E}_1(t).
%\end{aligned}
%\end{equation}
%We rewritten \eqref{C-30} as follows
\begin{equation*}
\begin{aligned}
\dot{\mathcal{E}}_1(t) &\le - \sqrt{K} \left\{ \mathcal{E}_1(t) - \left[ \frac{3(1+|E^c|r)D_0}{K^{\frac{3}{2}}a_l \cos \alpha \sin D_0}+ 12\frac{m}{K^{\frac{1}{2}}} \right] D_\Omega \right.\\
&\left.\qquad\qquad- 12N^2K^{\frac{1}{2}} a_u^2 \sin^2 \alpha \left[\frac{(1+|E^c|r)D_0}{a_l \cos \alpha \sin D_0}  + 4mK\right] \right\},
\end{aligned}
\end{equation*}
which yields that
\begin{equation}\label{C-31}
\mathcal{E}_1(t) \le \max \left\{ \mathcal{E}_1(0), 2C_1  \right\} < \frac{1}{8} D_0^2, \quad \text{for} \ t \in [0,T^*).
\end{equation}
Here, we use the assumption in \eqref{Con_16} that
\begin{equation*}
C_1 = \left[ \frac{3(1+|E^c|r)D_0}{K^{\frac{3}{2}}a_l \cos \alpha \sin D_0}+ 12\frac{m}{K^{\frac{1}{2}}} \right] D_\Omega + 12N^2K^{\frac{1}{2}} a_u^2 \sin^2 \alpha \left[\frac{(1+|E^c|r)D_0}{a_l \cos \alpha \sin D_0}  + 4mK\right] < \frac{1}{16} D_0^2.
\end{equation*}
Then it follows from \eqref{C-31} that
\begin{equation*}
\begin{aligned}
\mathcal{E}_1 (T^*) \le \max \left\{ \mathcal{E}_1(0), 2C_1  \right\} < \frac{1}{8} D_0^2,
\end{aligned}
\end{equation*}
which obviously contradicts to  \eqref{continuity}. This means $T^* = +\infty$. 
Thus, we obtain
\begin{equation}\label{C-32}
\mathcal{E}_1(t) < \frac{1}{8} D_0^2 \quad \text{and} \quad D(\theta(t)) < D_0 < \pi,  \ \forall \ 0 \le t < +\infty,
\end{equation}
which implies that
\begin{equation*}
\dot{\mathcal{E}}_1(t) \le \sqrt{K}C_1 - \sqrt{K} \mathcal{E}_1(t), \quad t \in [0,+\infty).
\end{equation*}
\end{proof}

\subsection{Entrance to a small region} Subsequently, we can use the dissipation of energy functional $\mathcal{E}_1$ to prove the decreasing of phase diameter.  

\begin{lemma}\label{p_bound}
Let $\theta(t) = (\theta_1(t), \theta_2(t), \cdots,\theta_N(t))$ be a solution to system \eqref{sp_dynamic} and \eqref{initial}, and suppose the assumptions in \eqref{Con_11}, \eqref{Con_13}, \eqref{Con_14}, \eqref{Con_16} hold.
Then, there exists a finite time $t_*$ such that
\begin{equation*}
D(\theta(t)) \le D^\infty, \quad t \ge t_*.
\end{equation*}
\end{lemma}
\begin{proof}
It is known from Lemma \ref{func1_dynamic} that for $t \in [0,+\infty)$,
\begin{equation}\label{C-33}
\begin{aligned}
\dot{\mathcal{E}}_1(t) &\le - \sqrt{K} \left\{ \mathcal{E}_1(t) - \left[ \frac{3(1+|E^c|r)D_0}{K^{\frac{3}{2}}a_l \cos \alpha \sin D_0}+ 12\frac{m}{K^{\frac{1}{2}}} \right] D_\Omega \right.\\
&\left.\qquad \qquad - 12N^2K^{\frac{1}{2}} a_u^2 \sin^2 \alpha \left[\frac{(1+|E^c|r)D_0}{a_l \cos \alpha \sin D_0}  + 4mK\right] \right\}.
\end{aligned}
\end{equation}
Without loss of generality, we assume $D_0 \in [\frac{\pi}{2}, \pi)$, otherwise, the desired result can follow from \eqref{C-32} with $D^\infty = D_0 $ and $t_* = 0$. We divide the proof into two cases.

\noindent $\bullet$ {\bf{Case 1:}} If $\mathcal{E}_1(0) \le \frac{1}{8} (D^\infty)^2$, based on \eqref{C-33} and the assumption \eqref{Con_16} that
\begin{equation*}
\left[ \frac{3(1+|E^c|r)D_0}{K^{\frac{3}{2}}a_l \cos \alpha \sin D_0}+ 12\frac{m}{K^{\frac{1}{2}}} \right] D_\Omega + 12N^2K^{\frac{1}{2}} a_u^2 \sin^2 \alpha \left[\frac{(1+|E^c|r)D_0}{a_l \cos \alpha \sin D_0}  + 4mK\right] < \frac{1}{16} (D^\infty)^2,
\end{equation*}
we have
\begin{equation*}
\begin{aligned}
\dot{\mathcal{E}}_1(t) \le - \sqrt{K} \left\{ \mathcal{E}_1(t) - \frac{1}{16} (D^\infty)^2 \right\}, \quad t \in [0,+\infty),
\end{aligned}
\end{equation*}
which implies that
\begin{equation*}
\mathcal{E}_1(t) \le \frac{1}{8} (D^\infty)^2, \quad t \in [0,+\infty).
\end{equation*}
According to Lemma \ref{p_positive}, this means that
\begin{equation*}
D(\theta(t)) \le \sqrt{8 \mathcal{E}_1(t)} \le D^\infty < \frac{\pi}{2}, \quad \text{for} \ t \in [0,+\infty),
\end{equation*}
which leads to the desired result with $t_* = 0$. 

\noindent $\bullet$ {\bf{Case 2:}} If $\mathcal{E}_1(0) > \frac{1}{8} (D^\infty)^2$, then for $\frac{1}{8} (D^\infty)^2 \le \mathcal{E}_1(t) \le \mathcal{E}_1(0)$, it follows from \eqref{C-33} and \eqref{Con_16} that
\begin{equation*}
\begin{aligned}
\dot{\mathcal{E}}_1(t) &\le - \sqrt{K} \left\{ \mathcal{E}_1(t) - \left[ \frac{3(1+|E^c|r)D_0}{K^{\frac{3}{2}}a_l \cos \alpha \sin D_0}+ 12\frac{m}{K^{\frac{1}{2}}} \right] D_\Omega \right.\\
&\left. \qquad\qquad- 12N^2K^{\frac{1}{2}} a_u^2 \sin^2 \alpha \left[\frac{(1+|E^c|r)D_0}{a_l \cos \alpha \sin D_0}  + 4mK\right] \right\}\\
&\le - \sqrt{K} \left[ \frac{1}{8} (D^\infty)^2 - \frac{1}{16} (D^\infty)^2\right] = - \frac{1}{16}\sqrt{K} (D^\infty)^2.
\end{aligned}
\end{equation*}
This means that we can find a finite time $t_*$ below
\begin{equation*}
t_* = \frac{\mathcal{E}_1(0) - \frac{1}{8} (D^\infty)^2 }{\frac{1}{16}\sqrt{K} (D^\infty)^2},
\end{equation*}
such that
\begin{equation}\label{C-34}
\begin{aligned}
\mathcal{E}_1(t) \le \frac{1}{8} (D^\infty)^2, \quad t \ge t_*. 
\end{aligned}
\end{equation}
Then, according to \eqref{C-34}, we have
\begin{equation*}
D(\theta(t)) \le D^\infty < \frac{\pi}{2}, \quad t \ge t_*,
\end{equation*}
which yields the desired result.
\end{proof}

\section{Frequency synchronization}\label{sec:4}
\setcounter{equation}{0}

Now we turn to study the frequency synchronization and the convergence rate. We will show that the frequency synchronization emerges exponentially fast for the second-order Kuramoto model  \eqref{sp_dynamic} with inertia and frustration. 

We directly differentiate second-order Kuramoto system \eqref{sp_dynamic} to obtain the dynamics of frequency as follows
\begin{equation}\label{sf_dynamic}
m \ddot{\omega}_i(t) + \dot{\omega}_i(t) = \frac{K}{N} \sum_{l=1}^N a_{il} \cos (\theta_l(t) - \theta_i(t) + \alpha) (\omega_l(t) - \omega_i(t)), \quad t\ge t_*, \ i=1,2,\ldots,N,
\end{equation}
where $t_*$ is given in Lemma \ref{p_bound}.
For the sake of discussion, we further apply the formula
\begin{equation*}
\cos(x+y) = \cos x \cos y - \sin x \sin y
\end{equation*}
to transform the system \eqref{sf_dynamic} into the form as below
\begin{equation}\label{ex_sf_dynamic}
\begin{aligned}
m \ddot{\omega}_i + \dot{\omega}_i &= \frac{K}{N} \sum_{l=1}^N a_{il} \left[ \cos (\theta_l - \theta_i)  \cos\alpha - \sin (\theta_l - \theta_i) \sin \alpha  \right]  (\omega_l - \omega_i) \\
&= \frac{K}{N} \cos \alpha \sum_{l=1}^N a_{il} \cos (\theta_l - \theta_i) (\omega_l - \omega_i) - \frac{K}{N} \sin \alpha \sum_{l=1}^N a_{il} \sin (\theta_l - \theta_i) (\omega_l - \omega_i). 
\end{aligned}
\end{equation}
Similar to Section \ref{sec:3}, we will observe the dissipation from the dynamics of a proper energy functional which is constructed as below:
\begin{equation}\label{f_functional}
\mathcal{E}_2(t) = 2m \sum_{i=1}^N \sum_{j=1}^N (\omega_i - \omega_j)(\dot{\omega}_i - \dot{\omega}_j) + (1 - m\sqrt{K})\sum_{i=1}^N \sum_{j=1}^N (\omega_i - \omega_j)^2 + 2m^2\sum_{i=1}^N \sum_{j=1}^N (\dot{\omega}_i - \dot{\omega}_j)^2.
\end{equation}

\subsection{Preparatory lemmas} In this part, we provide several energy estimates which will be crucially used in the later analysis. 
\begin{lemma}\label{f_positive}
Suppose the inertia and coupling strength satisfy
\begin{equation*}
m\sqrt{K} < \frac{1}{8}.
\end{equation*}
Then for any $t \ge 0$, we have
\begin{equation*}
\mathcal{E}_2(t)  \ge \frac{1}{8} \sum_{i=1}^N\sum_{j=1}^N(\omega_i(t) - \omega_j(t))^2 + \frac{2}{3} m^2 \sum_{i=1}^N \sum_{j=1}^N (\dot{\omega}_i(t) - \dot{\omega}_j(t))^2.
\end{equation*}
Moreover, we have
\begin{align*}
D(\omega(t)) \le \sqrt{8\mathcal{E}_2(t)}, \quad t\ge 0.
\end{align*}
\end{lemma}
\begin{proof}
The proof is similar to that in Lemma \ref{p_positive}.
\end{proof}
%Based on the formula
%\begin{equation*}
%2|a||b| \le a^2 + b^2, \quad a, b \in \mathbb{R},
%\end{equation*}
%we have
%\begin{equation*}
%2m \sum_{i=1}^N \sum_{j=1}^N (\omega_i(t) - \omega_j(t)) (\dot{\omega}_i(t) - \dot{\omega}_j(t))  \ge - \sum_{i=1}^N \sum_{j=1}^N [\frac{3}{4} (\omega_i(t) - \omega_j(t))^2 + \frac{4}{3} m^2 (\dot{\omega}_i(t) - \dot{\omega}_j(t))^2].
%\end{equation*}
%Then we see from \eqref{f_functional} that
%\begin{equation*}
%\begin{aligned}
%\mathcal{E}_2(t) &\ge \left(  \frac{1}{4} - m\sqrt{K} \right) \sum_{i=1}^N\sum_{j=1}^N(\omega_i(t) - \omega_j(t))^2 + \frac{2}{3} m^2 \sum_{i=1}^N \sum_{j=1}^N (\dot{\omega}_i(t) - \dot{\omega}_j(t))^2\\
%&\ge \frac{1}{8} \sum_{i=1}^N\sum_{j=1}^N(\omega_i(t) - \omega_j(t))^2 + \frac{2}{3} m^2 \sum_{i=1}^N \sum_{j=1}^N (\dot{\omega}_i(t) - \dot{\omega}_j(t))^2.
%\end{aligned}
%\end{equation*}

Recall that in Section \ref{sec:3}, we have shown that all Kuramoto oscillators will be concentrated into an arc with length less than $\frac{\pi}{2}$ after some finite time. Based on this good property, $\cos(\theta_i-\theta_l)$ is positive which eventually yields the dissipation of relative frequencies. In the following, we give some preparatory lemmas.

\begin{lemma}\label{E2_f} 
Suppose the graph $\mathcal{G}$ is symmetric and connected, and let $(\theta(t),\omega(t))$ be a solution to system \eqref{ex_sf_dynamic} satisfying
\begin{equation*}
D(\theta(t)) \le D^\infty < \frac{\pi}{2}.
\end{equation*}
Then, we have the following estimates:

\begin{align}
&- \frac{K}{N} \sin \alpha \sum_{i=1}^N \sum_{j=1}^N \sum_{l=1}^N [a_{il} \sin (\theta_l - \theta_i) (\omega_l - \omega_i) - a_{jl} \sin (\theta_l - \theta_j) (\omega_l - \omega_j)] \cdot 2(\omega_i - \omega_j)\label{F-1}\\
&\le 4 K a_u \sin \alpha \sin D^\infty \sum_{i=1}^N \sum_{j=1}^N (\omega_i - \omega_j)^2,\notag\\
&\notag\\
&\frac{K}{N} \cos \alpha \sum_{i=1}^N \sum_{j=1}^N  \sum_{l=1}^N [a_{il} \cos (\theta_l - \theta_i) (\omega_l - \omega_i) - a_{jl} \cos (\theta_l - \theta_j) (\omega_l - \omega_j)] \cdot 2(\omega_i - \omega_j)\label{F-3}\\
&\le - 2K a_l \cos \alpha  \cos D^\infty \frac{1}{1+ r|E^c|} \sum_{i=1}^N \sum_{j=1}^N (\omega_i - \omega_j)^2.\notag
\end{align}

\end{lemma}
\begin{proof}
\textbf{(1)} We first prove the first inequality. It can be seen from the assumption $D(\theta(t)) \le D^\infty$ that
\begin{equation*}
\begin{aligned}
&- \frac{K}{N} \sin \alpha \sum_{i=1}^N \sum_{j=1}^N \sum_{l=1}^N [a_{il} \sin (\theta_l - \theta_i) (\omega_l - \omega_i) - a_{jl} \sin (\theta_l - \theta_j) (\omega_l - \omega_j)] \cdot 2(\omega_i - \omega_j)\\
&\le \frac{K}{N} a_u \sin \alpha \sum_{i=1}^N \sum_{j=1}^N \sum_{l=1}^N \left[ |\sin (\theta_l - \theta_i)| |\omega_l - \omega_i| + |\sin (\theta_l - \theta_j)||\omega_l - \omega_j|\right] \cdot 2|\omega_i - \omega_j|\\
&\le \frac{K}{N} a_u \sin \alpha \sin D^\infty \sum_{i=1}^N \sum_{j=1}^N \sum_{l=1}^N [(\omega_l - \omega_i)^2 + (\omega_i - \omega_j)^2] \\
&+ \frac{K}{N} a_u \sin \alpha \sin D^\infty \sum_{i=1}^N \sum_{j=1}^N \sum_{l=1}^N [(\omega_l - \omega_j)^2+ (\omega_i - \omega_j)^2].
\end{aligned}
\end{equation*}
This yields that
\begin{equation}\label{F-2}
\begin{aligned}
&- \frac{K}{N} \sin \alpha \sum_{i=1}^N \sum_{j=1}^N \sum_{l=1}^N [a_{il} \sin (\theta_l - \theta_i) (\omega_l - \omega_i) - a_{jl} \sin (\theta_l - \theta_j) (\omega_l - \omega_j)] \cdot 2(\omega_i - \omega_j)\\
&\le 4 K a_u \sin \alpha \sin D^\infty \sum_{i=1}^N \sum_{j=1}^N (\omega_i - \omega_j)^2.
\end{aligned}
\end{equation}

\noindent \textbf{(2)} For the proof of second inequality, we set
\begin{equation*}
\begin{aligned}
\mathcal{L}_2 &= \frac{K}{N} \cos \alpha \sum_{i=1}^N \sum_{j=1}^N  \sum_{l=1}^N [a_{il} \cos (\theta_l - \theta_i) (\omega_l - \omega_i) - a_{jl} \cos (\theta_l - \theta_j) (\omega_l - \omega_j)] \cdot 2(\omega_i - \omega_j)\\
&=- \frac{4K}{N} \cos \alpha \sum^N_{\substack{i,j=1 \\ i \ne j}} a_{ji} \cos (\theta_i - \theta_j) (\omega_i - \omega_j)^2 \\
&+ \frac{K}{N} \cos \alpha\sum^N_{\substack{i,j=1 \\ i \ne j}} \sum^N_{\substack{l=1 \\ l \ne i,j}} [a_{il} \cos (\theta_l - \theta_i) (\omega_l - \omega_i)] \cdot 2(\omega_i - \omega_j)\\
&+\frac{K}{N} \cos \alpha\sum^N_{\substack{i,j=1 \\ i \ne j}} \sum^N_{\substack{l=1 \\ l \ne i,j}} [- a_{jl} \cos (\theta_l - \theta_j) (\omega_l - \omega_j)] \cdot 2(\omega_i - \omega_j).
\end{aligned}
\end{equation*}
%We further transform $\mathcal{L}_2$ into the following form
%\begin{equation}\label{F-4}
%\begin{aligned}
%\mathcal{L}_2 &= \frac{K}{N} \cos \alpha \sum^N_{\substack{i,j=1 \\ i \ne j}} \sum^N_{l=1} [a_{il} \cos (\theta_l - \theta_i) (\omega_l - \omega_i) - a_{jl} \cos (\theta_l - \theta_j) (\omega_l - \omega_j)] \cdot 2(\omega_i - \omega_j) \\
%&= \frac{K}{N} \cos \alpha \left[ \sum^N_{\substack{i,j=1 \\ i \ne j}} \sum_{l=i,j} + \sum^N_{\substack{i,j=1 \\ i \ne j}} \sum^N_{\substack{l=1 \\ l \ne i,j}}  \right] [a_{il} \cos (\theta_l - \theta_i) (\omega_l - \omega_i) - a_{jl} \cos (\theta_l - \theta_j) (\omega_l - \omega_j)] \cdot 2(\omega_i - \omega_j) \\
%&= - \frac{4K}{N} \cos \alpha \sum^N_{\substack{i,j=1 \\ i \ne j}} a_{ji} \cos (\theta_i - \theta_j) (\omega_i - \omega_j)^2 \\
%&+ \frac{K}{N} \cos \alpha\sum^N_{\substack{i,j=1 \\ i \ne j}} \sum^N_{\substack{l=1 \\ l \ne i,j}} [a_{il} \cos (\theta_l - \theta_i) (\omega_l - \omega_i)] \cdot 2(\omega_i - \omega_j)\\
%&+\frac{K}{N} \cos \alpha\sum^N_{\substack{i,j=1 \\ i \ne j}} \sum^N_{\substack{l=1 \\ l \ne i,j}} [- a_{jl} \cos (\theta_l - \theta_j) (\omega_l - \omega_j)] \cdot 2(\omega_i - \omega_j)\\ 
%\end{aligned}
%\end{equation}
For the last two terms in above equality, similar to the proof in Lemma \ref{E1_p}, we exchange the indices to have
\begin{equation*}
\begin{aligned}
&\frac{K}{N} \cos \alpha\sum^N_{\substack{i,j=1 \\ i \ne j}} \sum^N_{\substack{l=1 \\ l \ne i,j}} [a_{il} \cos (\theta_l - \theta_i) (\omega_l - \omega_i)] \cdot 2(\omega_i - \omega_j)\\
&+\frac{K}{N} \cos \alpha\sum^N_{\substack{i,j=1 \\ i \ne j}} \sum^N_{\substack{l=1 \\ l \ne i,j}} [- a_{jl} \cos (\theta_l - \theta_j) (\omega_l - \omega_j)] \cdot 2(\omega_i - \omega_j)\\
%&= \frac{K}{N} \cos \alpha\sum^N_{\substack{i,l=1 \\ i \ne l}} \sum^N_{\substack{j=1 \\ j \ne i,l}} [a_{ij}\cos (\theta_j - \theta_i)(\omega_j - \omega_i)]\cdot 2(\omega_i - \omega_l) \\
%&+ \frac{K}{N} \cos \alpha\sum^N_{\substack{l,j=1 \\ l \ne j}} \sum^N_{\substack{i=1 \\ l \ne l,j}}[- a_{ji} \cos (\theta_i  - \theta_j) (\omega_i - \omega_j)]\cdot 2(\omega_l - \omega_j)\\
%&= - \frac{2K}{N} \cos \alpha \sum^N_{\substack{i,j=1 \\ i \ne j}} \sum^N_{\substack{l=1\\ l \ne i,j}} a_{ji} \cos (\theta_i - \theta_j)(\omega_i - \omega_j)^2
&= - \frac{2K(N-2)}{N} \cos \alpha \sum^N_{\substack{i,j=1 \\ i \ne j}} a_{ji} \cos (\theta_i - \theta_j)(\omega_i - \omega_j)^2\\
%&= - \frac{K}{N} \cos \alpha \sum^N_{\substack{i,l=1 \\ i \ne l}} \sum^N_{\substack{j=1 \\ j \ne i,l}} [a_{ji} \cos (\theta_i - \theta_j)(\omega_i - \omega_j)] \cdot 2 (\omega_i - \omega_l)\\
%&- \frac{K}{N} \cos \alpha\sum^N_{\substack{l,j=1 \\ l \ne j}} \sum^N_{\substack{i=1 \\ l \ne l,j}}[a_{ji} \cos (\theta_i  - \theta_j) (\omega_i - \omega_j)]\cdot 2(\omega_l - \omega_j)\\
\end{aligned}
\end{equation*}
This together with Lemma \ref{pd_equiv} leads to
\begin{equation*}
\begin{aligned}
\mathcal{L}_2 &= - \frac{4K}{N} \cos \alpha \sum^N_{\substack{i,j=1 \\ i \ne j}} a_{ji} \cos (\theta_i - \theta_j) (\omega_i - \omega_j)^2 - \frac{2K(N-2)}{N} \cos \alpha \sum^N_{\substack{i,j=1 \\ i \ne j}} a_{ji} \cos (\theta_i - \theta_j)(\omega_i - \omega_j)^2\\
&= -2K \cos \alpha \sum_{i=1}^N \sum_{j=1}^N a_{ji} \cos (\theta_i - \theta_j)(\omega_i - \omega_j)^2 = - 2K \cos \alpha \sum_{(j,i) \in E} a_{ji} \cos (\theta_i - \theta_j)(\omega_i - \omega_j)^2\\
&\le - 2Ka_l \cos \alpha  \cos D^\infty \sum_{(j,i) \in E} (\omega_i - \omega_j)^2 \le - 2K a_l \cos \alpha  \cos D^\infty \frac{1}{1+ r|E^c|} \sum_{i=1}^N \sum_{j=1}^N (\omega_i - \omega_j)^2.
\end{aligned}
\end{equation*}
\end{proof}

\begin{lemma}\label{E2_dotf}
Suppose the graph $\mathcal{G}$ is symmetric and connected, and let $(\theta(t),\omega(t))$ be a solution to system  \eqref{ex_sf_dynamic} satisfying
\begin{equation*}
D(\theta(t)) \le D^\infty < \frac{\pi}{2}.
\end{equation*}
Then, we have the following estimates:

\begin{align}
&- \frac{2mK}{N} \sin \alpha \sum_{i=1}^N \sum_{j=1}^N\sum_{l=1}^N [a_{il} \sin (\theta_l - \theta_i) (\omega_l - \omega_i) - a_{jl} \sin (\theta_l - \theta_j) (\omega_l - \omega_j)] \cdot 2(\dot{\omega}_i - \dot{\omega}_j)\notag\\
&\le 4mK a_u \sin \alpha \sin D^\infty \sum_{i=1}^N \sum_{j=1}
^N (\omega_i - \omega_j)^2 + 4mK a_u \sin \alpha \sin D^\infty \sum_{i=1}^N \sum_{j=1}
^N (\dot{\omega}_i - \dot{\omega}_j)^2,\label{F-5} \\
&\notag\\
&\frac{2mK}{N} \cos \alpha \sum_{i=1}^N \sum_{j=1}^N\sum_{l=1}^N [a_{il} \cos (\theta_l - \theta_i) (\omega_l - \omega_i) - a_{jl} \cos (\theta_l - \theta_j) (\omega_l - \omega_j)] \cdot 2(\dot{\omega}_i - \dot{\omega}_j)\notag\\
&\le 8mK^2a_u^2 \cos^2 \alpha \sum_{i=1}^N \sum_{j=1}^N (\omega_i - \omega_j)^2 + \frac{m}{2} \sum_{i=1}^N \sum_{j=1}^N (\dot{\omega}_i - \dot{\omega}_j)^2.\label{F-6}
\end{align}

\end{lemma}
\begin{proof}
We will only prove \eqref{F-6} since \eqref{F-5} can be obtained by applying \eqref{formula1}. Now 
%\textbf{(1)} According to the assumption $D(\theta(t)) \le D^\infty < \frac{\pi}{2}$, we have
%\begin{equation*}
%\begin{aligned}
%&- \frac{2mK}{N} \sin \alpha \sum_{i=1}^N \sum_{j=1}^N\sum_{l=1}^N [a_{il} \sin (\theta_l - \theta_i) (\omega_l - \omega_i) - a_{jl} \sin (\theta_l - \theta_j) (\omega_l - \omega_j)] \cdot 2(\dot{\omega}_i - \dot{\omega}_j) \\
%&\le \frac{2mK}{N} a_u \sin \alpha \sin D^\infty \sum_{i=1}^N \sum_{j=1}^N\sum_{l=1}^N [|\omega_l - \omega_i| + |\omega_l - \omega_j|] \cdot 2|\dot{\omega}_i - \dot{\omega}_j| \\
%&\le \frac{2mK}{N} a_u \sin \alpha \sin D^\infty \sum_{i=1}^N \sum_{j=1}^N\sum_{l=1}^N [(\omega_l - \omega_i)^2 + (\dot{\omega}_i - \dot{\omega}_j)^2] \\
%&+ \frac{2mK}{N} a_u \sin \alpha \sin D^\infty \sum_{i=1}^N \sum_{j=1}^N\sum_{l=1}^N [(\omega_l - \omega_j)^2 + (\dot{\omega}_i - \dot{\omega}_j)^2].
%\end{aligned}
%\end{equation*}
%Thus, we obtain
%\begin{equation*}
%\begin{aligned}
%&- \frac{2mK}{N} \sin \alpha \sum_{i=1}^N \sum_{j=1}^N\sum_{l=1}^N [a_{il} \sin (\theta_l - \theta_i) (\omega_l - \omega_i) - a_{jl} \sin (\theta_l - \theta_j) (\omega_l - \omega_j)] \cdot 2(\dot{\omega}_i - \dot{\omega}_j) \\
%&\le 4mK a_u \sin \alpha \sin D^\infty \sum_{i=1}^N \sum_{j=1}
%^N (\omega_i - \omega_j)^2 + 4mK a_u \sin \alpha \sin D^\infty \sum_{i=1}^N \sum_{j=1}
%^N (\dot{\omega}_i - \dot{\omega}_j)^2. 
%\end{aligned}
%\end{equation*}
%\noindent \textbf{(2)} 
for convenience, we set
\begin{equation*}
\begin{aligned}
\mathcal{R}_2 = \frac{2mK}{N} \cos \alpha \sum_{i=1}^N \sum_{j=1}^N\sum_{l=1}^N [a_{il} \cos (\theta_l - \theta_i) (\omega_l - \omega_i) - a_{jl} \cos (\theta_l - \theta_j) (\omega_l - \omega_j)] \cdot 2(\dot{\omega}_i - \dot{\omega}_j).
\end{aligned}
\end{equation*}
Then we have 
\begin{equation}\label{F-7}
\begin{aligned}
\mathcal{R}_2 
%&= \frac{2mK}{N} \cos \alpha \sum^N_{\substack{i,j=1 \\ i \ne j}}  \sum_{l=1}^N [a_{il} \cos (\theta_l - \theta_i) (\omega_l - \omega_i) - a_{jl} \cos (\theta_l - \theta_j) (\omega_l - \omega_j)] \cdot 2(\dot{\omega}_i - \dot{\omega}_j)\\
%& = \frac{2mK}{N} \cos \alpha \left[ \sum^N_{\substack{i,j=1 \\ i \ne j}}  \sum_{l = i,j} + \sum^N_{\substack{i,j=1 \\ i \ne j}}  \sum^N_{\substack{l=1 \\ l \ne i,j}}\right] [a_{il} \cos (\theta_l - \theta_i) (\omega_l - \omega_i) - a_{jl} \cos (\theta_l - \theta_j) (\omega_l - \omega_j)] \cdot 2(\dot{\omega}_i - \dot{\omega}_j)\\
%&= \frac{K}{N} \cos \alpha \sum^N_{\substack{i,j=1 \\ i \ne j}} \left[ a_{ii}\cos(\theta_i - \theta_i)(\omega_i - \omega_i) - a_{ji}\cos(\theta_i - \theta_j)(\omega_i - \omega_j) \right.\\
%&\qquad \qquad   \left. + a_{ij} \cos (\theta_j - \theta_i) (\omega_j - \omega_i) - a_{jj}\cos (\theta_j - \theta_j) (\omega_j - \omega_j)\right] \cdot 2(\dot{\omega}_i - \dot{\omega}_j) \\
%&+  \frac{K}{N} \cos \alpha \sum^N_{\substack{i,j=1 \\ i \ne j}}  \sum^N_{\substack{l=1 \\ l \ne i,j}} [a_{il} \cos (\theta_l - \theta_i) (\omega_l - \omega_i) - a_{jl} \cos (\theta_l - \theta_j) (\omega_l - \omega_j)] \cdot 2(\dot{\omega}_i - \dot{\omega}_j) \\
&= - \frac{4mK}{N} \cos \alpha \sum^N_{\substack{i,j=1 \\ i \ne j}} a_{ji} \cos (\theta_i - \theta_j) (\omega_i - \omega_j) \cdot 2(\dot{\omega}_i - \dot{\omega}_j) \\
&+ \frac{2mK}{N} \cos \alpha \sum^N_{\substack{i,j=1 \\ i \ne j}}  \sum^N_{\substack{l=1 \\ l \ne i,j}} [a_{il} \cos (\theta_l - \theta_i) (\omega_l - \omega_i)] \cdot 2(\dot{\omega}_i - \dot{\omega}_j) \\
&+  \frac{2mK}{N} \cos \alpha \sum^N_{\substack{i,j=1 \\ i \ne j}}  \sum^N_{\substack{l=1 \\ l \ne i,j}} [- a_{jl} \cos (\theta_l - \theta_j) (\omega_l - \omega_j)] \cdot 2(\dot{\omega}_i - \dot{\omega}_j).\\
\end{aligned}
\end{equation}
Similar to the proof in Lemma \ref{E1_p}, we make exchanges of indices for the last two terms in above equality to obtain

\begin{equation*}
\begin{aligned}
& \frac{2mK}{N} \cos \alpha \sum^N_{\substack{i,j=1 \\ i \ne j}}  \sum^N_{\substack{l=1 \\ l \ne i,j}} [a_{il} \cos (\theta_l - \theta_i) (\omega_l - \omega_i)] \cdot 2(\dot{\omega}_i - \dot{\omega}_j) \\
&+  \frac{2mK}{N} \cos \alpha \sum^N_{\substack{i,j=1 \\ i \ne j}}  \sum^N_{\substack{l=1 \\ l \ne i,j}} [- a_{jl} \cos (\theta_l - \theta_j) (\omega_l - \omega_j)] \cdot 2(\dot{\omega}_i - \dot{\omega}_j)\\
%&= \frac{2mK}{N} \cos \alpha \sum^N_{\substack{i,l=1 \\ i \ne l}}  \sum^N_{\substack{j=1 \\ j \ne i,l}} [a_{ij} \cos (\theta_j - \theta_i) (\omega_j - \omega_i)] \cdot 2(\dot{\omega}_i - \dot{\omega}_l) \\
%&+  \frac{2mK}{N} \cos \alpha \sum^N_{\substack{l,j=1 \\ l \ne j}}  \sum^N_{\substack{i=1 \\ i \ne l,j}} [- a_{ji} \cos (\theta_i - \theta_j) (\omega_i - \omega_j)] \cdot 2(\dot{\omega}_l - \dot{\omega}_j)\\
%%& = - \frac{K}{N} \cos \alpha \sum^N_{\substack{i,j=1 \\ i \ne j}} \sum^N_{\substack{l=1 \\ l \ne i,j}} [a_{ji} \cos (\theta_i - \theta_j) (\omega_i - \omega_j)] \cdot 2(\dot{\omega}_i - \dot{\omega}_l) \\
%%&- \frac{K}{N} \cos \alpha \sum^N_{\substack{i,j=1 \\ i \ne j}}  \sum^N_{\substack{l=1 \\ l \ne i,j}} [a_{ji} \cos (\theta_i - \theta_j) (\omega_i - \omega_j)] \cdot 2(\dot{\omega}_l - \dot{\omega}_j)\\
%&= - \frac{4mK}{N} \cos \alpha \sum^N_{\substack{i,j=1 \\ i \ne j}} \sum^N_{\substack{l=1 \\ l \ne i,j}} a_{ji} \cos (\theta_i - \theta_j) (\omega_i - \omega_j) \cdot (\dot{\omega}_i - \dot{\omega}_j)\\
&=- \frac{4mK(N-2)}{N} \cos \alpha \sum^N_{\substack{i,j=1 \\ i \ne j}} a_{ji} \cos (\theta_i - \theta_j) (\omega_i - \omega_j) \cdot (\dot{\omega}_i - \dot{\omega}_j).
\end{aligned}
\end{equation*}
This together with \eqref{F-7} implies that
 \begin{equation}\label{E-16}
 \begin{aligned}
 \mathcal{R}_2 
% &= - \frac{4mK}{N} \cos \alpha \sum^N_{\substack{i,j=1 \\ i \ne j}} a_{ji} \cos (\theta_i - \theta_j) (\omega_i - \omega_j) \cdot 2(\dot{\omega}_i - \dot{\omega}_j)\\
% &- \frac{4mK(N-2)}{N} \cos \alpha \sum^N_{\substack{i,j=1 \\ i \ne j}} a_{ji} \cos (\theta_i - \theta_j) (\omega_i - \omega_j) \cdot (\dot{\omega}_i - \dot{\omega}_j)\\
% & = -4mK \cos \alpha \sum^N_{\substack{i,j=1 \\ i \ne j}} a_{ji} \cos (\theta_i - \theta_j) (\omega_i - \omega_j) \cdot (\dot{\omega}_i - \dot{\omega}_j)\\
 & =  -4mK \cos \alpha \sum_{i=1}^N \sum_{j=1}^N a_{ji} \cos (\theta_i - \theta_j) (\omega_i - \omega_j) \cdot (\dot{\omega}_i - \dot{\omega}_j).\\
% &\le K a_u \cos \alpha \sum_{i=1}^N \sum_{j=1}^N [(\omega_i - \omega_j)^2 + (\dot{\omega}_i - \dot{\omega}_j)^2] \\
% & = K a_u \cos \alpha \sum_{i=1}^N \sum_{j=1}^N (\omega_i - \omega_j)^2 + K a_u \cos \alpha \sum_{i=1}^N \sum_{j=1}^N (\dot{\omega}_i - \dot{\omega}_j)^2.
 \end{aligned}
 \end{equation}
Therefore, we eventually obtain the following estimate:
\begin{equation*}
\begin{aligned}
\mathcal{R}_2 
 &\le 4mK a_u \cos \alpha \sum_{i=1}^N \sum_{j=1}^N  |\omega_i - \omega_j| |\dot{\omega}_i - \dot{\omega}_j| \\
 &= \sum_{i=1}^N \sum_{j=1}^N 2\cdot 2mK a_u \cos \alpha \sqrt{\frac{2}{m}} |\omega_i - \omega_j| \cdot \sqrt{\frac{m}{2}} |\dot{\omega}_i - \dot{\omega}_j|\\
 &\le 8mK^2a_u^2 \cos^2 \alpha \sum_{i=1}^N \sum_{j=1}^N (\omega_i - \omega_j)^2 + \frac{m}{2} \sum_{i=1}^N \sum_{j=1}^N (\dot{\omega}_i - \dot{\omega}_j)^2.\\
\end{aligned}
\end{equation*}
\end{proof}

\subsection{Complete synchronization}
Next, we will show that the frequency diameter exponentially decays to zero for the Kuramoto model \eqref{sp_dynamic} with inertia and frustration.

\begin{lemma}\label{f_expo}
Let $(\theta(t), \omega(t))$ be a solution to system \eqref{ex_sf_dynamic} and suppose the initial data and system parameters satisfy the Assumption $(\mathcal{A})$.
%\begin{equation}\label{Con_4}
%\begin{aligned}
%&\sin \alpha < \frac{a_l \cos \alpha \cos D^\infty}{12a_u (1+ r|E^c|)\sin D^\infty}, \quad m \sin \alpha <  \frac{a_l \cos \alpha \cos D^\infty}{12 a_u (1 + r|E^c|)\sin D^\infty} \  \Longrightarrow m< 1,\\
%&mK < \frac{a_l  \cos D^\infty}{24 a_u^2 \cos\alpha (1+ r|E^c|)},\quad K\sin \alpha < \frac{1}{8a_u\sin D^\infty},\\
%&\sqrt{K} a_l \cos \alpha  \cos D^\infty \frac{1}{1+ r|E^c|} > 1 > 1-m\sqrt{K}, \quad \frac{m}{\sqrt{K}} > 2m^2 \ \Longrightarrow \ m\sqrt{K} < \frac{1}{2}.
%\end{aligned}
%\end{equation}
Then, we have
\begin{equation*}
\mathcal{E}_2(t) \le \mathcal{E}_2(t_*) e^{-\sqrt{K}(t - t_*)}, \quad t \ge t_*.
\end{equation*} 
%\begin{equation*}
%\dot{\mathcal{E}}_2(t) \le -\sqrt{K} \mathcal{E}_2(t), \quad t \ge t_*.
%\end{equation*}
Moreover, we have
\begin{equation*}
D(\omega(t)) = \max_{1 \le i ,j \le N}|\omega_i(t) - \omega_j(t)| \le \sqrt{8\mathcal{E}_2(t_*)} e^{- \frac{\sqrt{K}}{2}(t - t_*)}, \quad t \ge t_*,
\end{equation*}
where $t_*$ is given in Lemma \ref{p_bound}.
\end{lemma}
\begin{proof}
We split the proof into three steps.

\noindent $\bullet$ \textbf{Step 1:} According to \eqref{ex_sf_dynamic}, we have
%it follows that for any $1 \le i,j \le N$,
%\begin{equation*}
%\begin{aligned}
%&m \ddot{\omega}_i + \dot{\omega}_i\\
%&= \frac{K}{N} \cos \alpha \sum_{l=1}^N a_{il} \cos (\theta_l - \theta_i) (\omega_l - \omega_i) - \frac{K}{N} \sin \alpha \sum_{l=1}^N a_{il} \sin (\theta_l - \theta_i) (\omega_l - \omega_i), 
%\end{aligned}
%\end{equation*}
%and
%\begin{equation*}
%\begin{aligned}
%&m \ddot{\omega}_j + \dot{\omega}_j\\
%&= \frac{K}{N} \cos \alpha \sum_{l=1}^N a_{jl} \cos (\theta_l - \theta_j) (\omega_l - \omega_j) - \frac{K}{N} \sin \alpha \sum_{l=1}^N a_{jl} \sin (\theta_l - \theta_j) (\omega_l - \omega_j). 
%\end{aligned}
%\end{equation*}
%We make a subtraction for above two equalities
\begin{equation}\label{E-1}
\begin{aligned}
m (\ddot{\omega}_i - \ddot{\omega}_j) + (\dot{\omega}_i - \dot{\omega}_j) &= \frac{K}{N} \cos \alpha \sum_{l=1}^N [a_{il} \cos (\theta_l - \theta_i) (\omega_l - \omega_i) - a_{jl} \cos (\theta_l - \theta_j) (\omega_l - \omega_j)]\\
&- \frac{K}{N} \sin \alpha \sum_{l=1}^N [a_{il} \sin (\theta_l - \theta_i) (\omega_l - \omega_i) - a_{jl} \sin (\theta_l - \theta_j) (\omega_l - \omega_j)].
\end{aligned}
\end{equation}
Next, we multiply \eqref{E-1} by $2(\omega_i - \omega_j)$ to get
%\begin{equation}\label{E-2}
%\begin{aligned}
%&m (\ddot{\omega}_i - \ddot{\omega}_j) \cdot 2(\omega_i - \omega_j) + (\dot{\omega}_i - \dot{\omega}_j) \cdot 2(\omega_i- \omega_j)\\
%&=  \frac{K}{N} \cos \alpha \sum_{l=1}^N [a_{il} \cos (\theta_l - \theta_i) (\omega_l - \omega_i) - a_{jl} \cos (\theta_l - \theta_j) (\omega_l - \omega_j)] \cdot 2(\omega_i - \omega_j)\\
%&- \frac{K}{N} \sin \alpha \sum_{l=1}^N [a_{il} \sin (\theta_l - \theta_i) (\omega_l - \omega_i) - a_{jl} \sin (\theta_l - \theta_j) (\omega_l - \omega_j)] \cdot 2(\omega_i - \omega_j).
%\end{aligned}
%\end{equation} 
%Since
%\begin{equation*}
%\begin{aligned}
%&\frac{d}{dt} (\omega_i - \omega_j)^2 = 2(\omega_i - \omega_j) (\dot{\omega}_i - \dot{\omega}_j),\quad \frac{d^2}{dt^2} (\omega_i - \omega_j)^2 = 2(\dot{\omega}_i - \dot{\omega}_j)^2 + 2(\omega_i - \omega_j)(\ddot{\omega}_i - \ddot{\omega}_j),
%\end{aligned}
%\end{equation*}
%it can be seen from \eqref{E-2} that
\begin{equation}\label{E-3}
\begin{aligned}
&m \left[\frac{d^2}{dt^2} (\omega_i - \omega_j)^2 - 2(\dot{\omega}_i - \dot{\omega}_j)^2  \right] + \frac{d}{dt} (\omega_i - \omega_j)^2 \\
&=  \frac{K}{N} \cos \alpha \sum_{l=1}^N [a_{il} \cos (\theta_l - \theta_i) (\omega_l - \omega_i) - a_{jl} \cos (\theta_l - \theta_j) (\omega_l - \omega_j)] \cdot 2(\omega_i - \omega_j)\\
&- \frac{K}{N} \sin \alpha \sum_{l=1}^N [a_{il} \sin (\theta_l - \theta_i) (\omega_l - \omega_i) - a_{jl} \sin (\theta_l - \theta_j) (\omega_l - \omega_j)] \cdot 2(\omega_i - \omega_j).
\end{aligned}
\end{equation}
Moreover, based on Lemma \ref{p_bound}, we sum \eqref{E-3} over $i,j$ and further apply the estimates \eqref{F-1} and \eqref{F-3} in Lemma \ref{E2_f} to obtain that for $t \ge t_*$,
\begin{equation}\label{E-4}
\begin{aligned}
&m \left[\frac{d^2}{dt^2} \sum_{i=1}^N \sum_{j=1}^N (\omega_i - \omega_j)^2 - 2\sum_{i=1}^N \sum_{j=1}^N(\dot{\omega}_i - \dot{\omega}_j)^2  \right] + \frac{d}{dt} \sum_{i=1}^N \sum_{j=1}^N (\omega_i - \omega_j)^2 \\
%&=- \frac{K}{N} \sin \alpha \sum_{i=1}^N \sum_{j=1}^N \sum_{l=1}^N [a_{il} \sin (\theta_l - \theta_i) (\omega_l - \omega_i) - a_{jl} \sin (\theta_l - \theta_j) (\omega_l - \omega_j)] \cdot 2(\omega_i - \omega_j)\\
%&+ \frac{K}{N} \cos \alpha \sum_{i=1}^N \sum_{j=1}^N  \sum_{l=1}^N [a_{il} \cos (\theta_l - \theta_i) (\omega_l - \omega_i) - a_{jl} \cos (\theta_l - \theta_j) (\omega_l - \omega_j)] \cdot 2(\omega_i - \omega_j)\\
&\le 4 K a_u \sin \alpha \sin D^\infty \sum_{i=1}^N \sum_{j=1}^N (\omega_i - \omega_j)^2 - 2K a_l \cos \alpha  \cos D^\infty \frac{1}{1+ r|E^c|} \sum_{i=1}^N \sum_{j=1}^N (\omega_i - \omega_j)^2.
\end{aligned}
\end{equation}

\noindent \textbf{Step 2:} Next, we multiply \eqref{E-1} by $2(\dot{\omega}_i - \dot{\omega}_j)$ on both sides to have
%\begin{equation}\label{E-10}
%\begin{aligned}
%&m (\ddot{\omega}_i - \ddot{\omega}_j) \cdot 2(\dot{\omega}_i - \dot{\omega}_j) + (\dot{\omega}_i - \dot{\omega}_j) \cdot 2(\dot{\omega}_i - \dot{\omega}_j)\\
%&=  \frac{K}{N} \cos \alpha \sum_{l=1}^N [a_{il} \cos (\theta_l - \theta_i) (\omega_l - \omega_i) - a_{jl} \cos (\theta_l - \theta_j) (\omega_l - \omega_j)] \cdot 2(\dot{\omega}_i - \dot{\omega}_j)\\
%&- \frac{K}{N} \sin \alpha \sum_{l=1}^N [a_{il} \sin (\theta_l - \theta_i) (\omega_l - \omega_i) - a_{jl} \sin (\theta_l - \theta_j) (\omega_l - \omega_j)] \cdot 2(\dot{\omega}_i - \dot{\omega}_j).
%\end{aligned}
%\end{equation}
%According to the fact
%\begin{equation*}
%\frac{d}{dt} (\dot{\omega}_i - \dot{\omega}_j)^2 = 2(\dot{\omega}_i - \dot{\omega}_j)(\ddot{\omega}_i - \ddot{\omega}_j),
%\end{equation*}
%it is easy to see from \eqref{E-10} that
\begin{equation}\label{E-11}
\begin{aligned}
&m \frac{d}{dt} (\dot{\omega}_i - \dot{\omega}_j)^2 + 2(\dot{\omega}_i - \dot{\omega}_j)^2 \\
&=  \frac{K}{N} \cos \alpha \sum_{l=1}^N [a_{il} \cos (\theta_l - \theta_i) (\omega_l - \omega_i) - a_{jl} \cos (\theta_l - \theta_j) (\omega_l - \omega_j)] \cdot 2(\dot{\omega}_i - \dot{\omega}_j)\\
&- \frac{K}{N} \sin \alpha \sum_{l=1}^N [a_{il} \sin (\theta_l - \theta_i) (\omega_l - \omega_i) - a_{jl} \sin (\theta_l - \theta_j) (\omega_l - \omega_j)] \cdot 2(\dot{\omega}_i - \dot{\omega}_j).
\end{aligned}
\end{equation}
We sum over $i,j$ and multiply by $2m$ on both sides of  \eqref{E-11}, then  
%\begin{equation*}
%\begin{aligned}
%&2m^2 \frac{d}{dt} \sum_{i=1}^N \sum_{j=1}^N (\dot{\omega}_i - \dot{\omega}_j)^2 + 4m\sum_{i=1}^N \sum_{j=1}^N(\dot{\omega}_i - \dot{\omega}_j)^2 \\
%&=- \frac{2mK}{N} \sin \alpha \sum_{i=1}^N \sum_{j=1}^N\sum_{l=1}^N [a_{il} \sin (\theta_l - \theta_i) (\omega_l - \omega_i) - a_{jl} \sin (\theta_l - \theta_j) (\omega_l - \omega_j)] \cdot 2(\dot{\omega}_i - \dot{\omega}_j)\\
%&+  \frac{2mK}{N} \cos \alpha \sum_{i=1}^N \sum_{j=1}^N\sum_{l=1}^N [a_{il} \cos (\theta_l - \theta_i) (\omega_l - \omega_i) - a_{jl} \cos (\theta_l - \theta_j) (\omega_l - \omega_j)] \cdot 2(\dot{\omega}_i - \dot{\omega}_j)\\
%\end{aligned}
%\end{equation*}
we exploit Lemma \ref{p_bound} and the estimates \eqref{F-5} and \eqref{F-6} in Lemma \ref{E2_dotf} to get
\begin{equation}\label{F-8}
\begin{aligned}
&2m^2 \frac{d}{dt} \sum_{i=1}^N \sum_{j=1}^N (\dot{\omega}_i - \dot{\omega}_j)^2 + 4m\sum_{i=1}^N \sum_{j=1}^N(\dot{\omega}_i - \dot{\omega}_j)^2 \\
&\le 4mK a_u \sin \alpha \sin D^\infty \sum_{i=1}^N \sum_{j=1}
^N (\omega_i - \omega_j)^2 + 4mK a_u \sin \alpha \sin D^\infty \sum_{i=1}^N \sum_{j=1}
^N (\dot{\omega}_i - \dot{\omega}_j)^2 \\
&+ 8mK^2a_u^2 \cos^2 \alpha \sum_{i=1}^N \sum_{j=1}^N (\omega_i - \omega_j)^2 + \frac{m}{2} \sum_{i=1}^N \sum_{j=1}^N (\dot{\omega}_i - \dot{\omega}_j)^2, \quad t \ge t_*.
\end{aligned}
\end{equation}
\noindent \textbf{Step 3:} Now we add \eqref{E-4} and \eqref{F-8} together to deduce
%\begin{equation*}
%\begin{aligned}
%&m \frac{d^2}{dt^2} \sum_{i=1}^N \sum_{j=1}^N (\omega_i - \omega_j)^2 +\frac{d}{dt} \sum_{i=1}^N \sum_{j=1}^N (\omega_i - \omega_j)^2 + 2m^2\frac{d}{dt} \sum_{i=1}^N \sum_{j=1}^N (\dot{\omega}_i - \dot{\omega}_j)^2 + 2m \sum_{i=1}^N \sum_{j=1}^N(\dot{\omega}_i - \dot{\omega}_j)^2 \\
%&\le 4 K a_u \sin \alpha \sin D^\infty \sum_{i=1}^N \sum_{j=1}^N (\omega_i - \omega_j)^2 - 2K a_l \cos \alpha  \cos D^\infty \frac{1}{1+ r|E^c|} \sum_{i=1}^N \sum_{j=1}^N (\omega_i - \omega_j)^2\\
%&+4mK a_u \sin \alpha \sin D^\infty \sum_{i=1}^N \sum_{j=1}
%^N (\omega_i - \omega_j)^2 + 4mK a_u \sin \alpha \sin D^\infty \sum_{i=1}^N \sum_{j=1}
%^N (\dot{\omega}_i - \dot{\omega}_j)^2 \\
%&+ 8mK^2a_u^2 \cos^2 \alpha \sum_{i=1}^N \sum_{j=1}^N (\omega_i - \omega_j)^2 + \frac{m}{2} \sum_{i=1}^N \sum_{j=1}^N (\dot{\omega}_i - \dot{\omega}_j)^2\\
%\end{aligned}
%\end{equation*}
 that for $t \ge t_*$,
\begin{equation}\label{E-19}
\begin{aligned}
&\frac{d}{dt}\left[ m\frac{d}{dt} \sum_{i=1}^N \sum_{j=1}^N (\omega_i - \omega_j)^2 + (1 - m\sqrt{K})\sum_{i=1}^N \sum_{j=1}^N (\omega_i - \omega_j)^2 + 2m^2\sum_{i=1}^N \sum_{j=1}^N (\dot{\omega}_i - \dot{\omega}_j)^2   \right]\\
%&\le 4 K a_u \sin \alpha \sin D^\infty \sum_{i=1}^N \sum_{j=1}^N (\omega_i - \omega_j)^2 + 4mK a_u \sin \alpha \sin D^\infty \sum_{i=1}^N \sum_{j=1}
%^N (\omega_i - \omega_j)^2\\
%&+8mK^2a_u^2 \cos^2 \alpha \sum_{i=1}^N \sum_{j=1}^N (\omega_i - \omega_j)^2+ 4mK a_u \sin \alpha \sin D^\infty \sum_{i=1}^N \sum_{j=1}
%^N (\dot{\omega}_i - \dot{\omega}_j)^2\\
%&- m\sqrt{K} \frac{d}{dt} \sum_{i=1}^N \sum_{j=1}^N (\omega_i - \omega_j)^2- 2K a_l \cos \alpha  \cos D^\infty \frac{1}{1+ r|E^c|} \sum_{i=1}^N \sum_{j=1}^N (\omega_i - \omega_j)^2- \frac{3}{2}m \sum_{i=1}^N \sum_{j=1}^N(\dot{\omega}_i - \dot{\omega}_j)^2\\
&\leq \mathcal{S}_1 + \mathcal{S}_2 + \mathcal{S}_3 + \mathcal{S}_4 - m\sqrt{K} \frac{d}{dt} \sum_{i=1}^N \sum_{j=1}^N (\omega_i - \omega_j)^2 \\
&- 2K a_l \cos \alpha  \cos D^\infty \frac{1}{1+ r|E^c|} \sum_{i=1}^N \sum_{j=1}^N (\omega_i - \omega_j)^2- \frac{3}{2}m \sum_{i=1}^N \sum_{j=1}^N(\dot{\omega}_i - \dot{\omega}_j)^2,\\
\end{aligned}
\end{equation}
where 
\begin{align*}
	&\mathcal{S}_1=4 K a_u \sin \alpha \sin D^\infty \sum_{i=1}^N \sum_{j=1}^N (\omega_i - \omega_j)^2,\quad \mathcal{S}_2= 4mK a_u \sin \alpha \sin D^\infty \sum_{i=1}^N \sum_{j=1}
     ^N (\omega_i - \omega_j)^2,\\
    &\mathcal{S}_3= 8mK^2a_u^2 \cos^2 \alpha \sum_{i=1}^N \sum_{j=1}^N (\omega_i - \omega_j)^2,\quad \mathcal{S}_4=4mK a_u \sin \alpha \sin D^\infty \sum_{i=1}^N \sum_{j=1}
^N (\dot{\omega}_i - \dot{\omega}_j)^2.
\end{align*}

%
%
%In the following, we will make some estimates on $\mathcal{S}_1,\mathcal{S}_2,\mathcal{S}_3,\mathcal{S}_4$ respectively.

\noindent $\diamond$ {\bf{Estimate of $\mathcal{S}_1$:}} Under the assumption $\eqref{Con_12}_1$ that
\begin{equation*}
\sin \alpha < \frac{a_l \cos \alpha \cos D^\infty}{12a_u (1+ r|E^c|) \sin D^\infty},
\end{equation*}
we have
\begin{equation}\label{E-20}
\mathcal{S}_1= 4 K a_u \sin \alpha \sin D^\infty \sum_{i=1}^N \sum_{j=1}^N (\omega_i - \omega_j)^2 < \frac{1}{3} K a_l \cos \alpha  \cos D^\infty \frac{1}{1+ r|E^c|} \sum_{i=1}^N \sum_{j=1}^N (\omega_i - \omega_j)^2.
\end{equation}

\noindent $\diamond$ {\bf{Estimate of $\mathcal{S}_2$:}} Owing to the assumption $\eqref{Con_12}$ that
\begin{equation*}
 \sin \alpha <  \frac{a_l \cos \alpha \cos D^\infty}{12 a_u (1 + r|E^c|)\sin D^\infty}, \quad m < 1,
\end{equation*}
it follows that
\begin{equation}\label{E-21}
\mathcal{S}_2  = 4mK a_u \sin \alpha \sin D^\infty \sum_{i=1}^N \sum_{j=1}
^N (\omega_i - \omega_j)^2 < \frac{1}{3} K a_l \cos \alpha  \cos D^\infty \frac{1}{1+ r|E^c|} \sum_{i=1}^N \sum_{j=1}^N (\omega_i - \omega_j)^2.
\end{equation}

\noindent $\diamond$ {\bf{Estimate of $\mathcal{S}_3$:}} According to the assumption $\eqref{Con_14}_2$ that
\begin{equation*}
mK < \frac{a_l  \cos D^\infty}{24 a_u^2 (1+ r|E^c|) \cos\alpha },
\end{equation*}
we have
\begin{equation}\label{E-22}
\begin{aligned}
\mathcal{S}_3& = 8mK^2a_u^2 \cos^2 \alpha \sum_{i=1}^N \sum_{j=1}^N (\omega_i - \omega_j)^2 < \frac{1}{3} K a_l \cos \alpha  \cos D^\infty \frac{1}{1+ r|E^c|} \sum_{i=1}^N \sum_{j=1}^N (\omega_i - \omega_j)^2 .
\end{aligned}
\end{equation}

\noindent $\diamond$ {\bf{Estimate of $\mathcal{S}_4$:}}
Due to the assumption \eqref{Con_15} that
\begin{equation*}
K\sin \alpha < \frac{1}{8a_u\sin D^\infty},
\end{equation*}
we obtain that 
\begin{equation}\label{E-23}
\begin{aligned}
\mathcal{S}_4 = 4mK a_u \sin \alpha \sin D^\infty \sum_{i=1}^N \sum_{j=1}
^N (\dot{\omega}_i - \dot{\omega}_j)^2 < \frac{m}{2} \sum_{i=1}^N \sum_{j=1}^N (\dot{\omega}_i - \dot{\omega}_j)^2.
\end{aligned}
\end{equation}

Therefore we combine \eqref{E-19}-\eqref{E-23} to get for $t \ge t_*$,
\begin{equation}\label{E-24}
\begin{aligned}
&\frac{d}{dt}\left[ m\frac{d}{dt} \sum_{i=1}^N \sum_{j=1}^N (\omega_i - \omega_j)^2 + (1 - m\sqrt{K})\sum_{i=1}^N \sum_{j=1}^N (\omega_i - \omega_j)^2 + 2m^2\sum_{i=1}^N \sum_{j=1}^N (\dot{\omega}_i - \dot{\omega}_j)^2   \right]\\
%&\le \frac{1}{3} K a_l \cos \alpha  \cos D^\infty \frac{1}{1+ r|E^c|} \sum_{i=1}^N \sum_{j=1}^N (\omega_i - \omega_j)^2 +  \frac{1}{3} K a_l \cos \alpha  \cos D^\infty \frac{1}{1+ r|E^c|} \sum_{i=1}^N \sum_{j=1}^N (\omega_i - \omega_j)^2\\
%&+ \frac{1}{3} K a_l \cos \alpha  \cos D^\infty \frac{1}{1+ r|E^c|} \sum_{i=1}^N \sum_{j=1}^N (\omega_i - \omega_j)^2 + \frac{m}{2} \sum_{i=1}^N \sum_{j=1}^N (\dot{\omega}_i - \dot{\omega}_j)^2 \\
%&- m\sqrt{K} \frac{d}{dt} \sum_{i=1}^N \sum_{j=1}^N (\omega_i - \omega_j)^2- 2K a_l \cos \alpha  \cos D^\infty \frac{1}{1+ r|E^c|} \sum_{i=1}^N \sum_{j=1}^N (\omega_i - \omega_j)^2- \frac{3}{2}m \sum_{i=1}^N \sum_{j=1}^N(\dot{\omega}_i - \dot{\omega}_j)^2\\
&\le - m\sqrt{K} \frac{d}{dt} \sum_{i=1}^N \sum_{j=1}^N (\omega_i - \omega_j)^2- K a_l \cos \alpha  \cos D^\infty \frac{1}{1+ r|E^c|} \sum_{i=1}^N \sum_{j=1}^N (\omega_i - \omega_j)^2 \\
&- m \sum_{i=1}^N \sum_{j=1}^N(\dot{\omega}_i - \dot{\omega}_j)^2.
\end{aligned}
\end{equation}
Moreover, applying the assumptions \eqref{Con_13} and $\eqref{Con_14}_1$ that
\begin{equation*}
\sqrt{K} a_l \cos \alpha  \cos D^\infty \frac{1}{1+ r|E^c|} > 1 > 1-m\sqrt{K}, \quad \frac{m}{\sqrt{K}} > 2m^2 .
\end{equation*}
we derive from \eqref{E-24} that for $t \ge t_*$,
\begin{equation*}
\begin{aligned}
&\frac{d}{dt}\left[ m\frac{d}{dt} \sum_{i=1}^N \sum_{j=1}^N (\omega_i - \omega_j)^2 + (1 - m\sqrt{K})\sum_{i=1}^N \sum_{j=1}^N (\omega_i - \omega_j)^2 + 2m^2\sum_{i=1}^N \sum_{j=1}^N (\dot{\omega}_i - \dot{\omega}_j)^2   \right]\\
&\le -\sqrt{K} \left[ m \frac{d}{dt} \sum_{i=1}^N \sum_{j=1}^N (\omega_i - \omega_j)^2 + \sqrt{K} a_l \cos \alpha  \cos D^\infty \frac{1}{1+ r|E^c|} \sum_{i=1}^N \sum_{j=1}^N (\omega_i - \omega_j)^2 \right. \\
&\left. \qquad \qquad +  \frac{m}{\sqrt{K}}\sum_{i=1}^N \sum_{j=1}^N(\dot{\omega}_i - \dot{\omega}_j)^2\right]\\
& \le -\sqrt{K} \left[ m \frac{d}{dt} \sum_{i=1}^N \sum_{j=1}^N (\omega_i - \omega_j)^2+(1 - m\sqrt{K})\sum_{i=1}^N \sum_{j=1}^N (\omega_i - \omega_j)^2 + 2m^2 \sum_{i=1}^N \sum_{j=1}^N(\dot{\omega}_i - \dot{\omega}_j)^2 \right].
\end{aligned}
\end{equation*}
Thus, it follows from \eqref{f_functional} that
\begin{equation*}
\dot{\mathcal{E}}_2(t) \le -\sqrt{K} \mathcal{E}_2(t), \quad t \ge t_*,
\end{equation*}
which immediately leads to
\begin{equation*}
\mathcal{E}_2(t) \le \mathcal{E}_2(t_*) e^{-\sqrt{K}(t - t_*)}, \quad t \ge t_*.
\end{equation*}
Then according to Lemma \ref{f_positive}, we ultimately obtain that
\begin{equation*}
D(\omega(t)) \le \sqrt{8\mathcal{E}_2(t)} \le \sqrt{8\mathcal{E}_2(t_*)} e^{- \frac{\sqrt{K}}{2}(t - t_*)}, \quad t \ge t_*, 
\end{equation*}
\end{proof}

Now, we are ready to provide the proof of Theorem \ref{main}.\newline

\noindent \textbf{Proof of Theorem \ref{main}:}
\begin{proof}
We combine Lemma \ref{p_bound} and Lemma \ref{f_expo} to complete the proof of Theorem \ref{main}.
\end{proof}

\section{Summary}\label{sec:5}
\setcounter{equation}{0}

We studied the dynamics of Kuramoto oscillators under the interplay of the effects of inertia and frustration under a connected and undirected graph. For this, we constructed two energy functionals which can bound the phase and frequency diameters, respectively. Under a sufficient regime in terms of initial data, small inertia, small frustration and large coupling, we first derived that the phase diameter will shrink to a small region after some finite time by the energy method. Next, we similarly applied the energy estimates and showed that the frequency diameter approaches to zero at the exponential rate. This ultimately yields the emergence of complete frequency synchronization. However, the interaction network we considered here is rather limited. In a general network, the energy functionals we constructed here become useless. In that case, we need to reconsider the dynamics of the diameters, and this will be our future work.

%\section*{Acknowledgments}
\section*{Declarations}

The work of T. Zhu is supported by the National Natural Science Foundation of China (Grant/Award Number: 12201172), the Natural Science Research Project of Universities in Anhui Province, China (Project Number: 2022AH051790) and  the Talent Research Fund of Hefei University, China (Grant/Award Number: 21-22RC23).

\end{document}